\documentclass[a4paper,12pt]{amsart}

\usepackage{hyperref}
\usepackage{microtype}
\usepackage{amssymb,amsmath,amsfonts,amsthm,bm,latexsym,amscd,verbatim,url,nicefrac,stmaryrd,enumerate,colonequals,dsfont,appendix}
\usepackage[all]{xy}

\usepackage{times}
\usepackage{graphicx}
\usepackage{fullpage}

{\bf}{\it}
\newtheorem*{thm*}{Theorem}
\newtheorem{thm}{Theorem}[section]{\bf}{\it}
\newtheorem{prop}[thm]{Proposition}
\newtheorem{lemma}[thm]{Lemma}
\newtheorem{cor}[thm]{Corollary}

\theoremstyle{definition}
\newtheorem{dfn}[thm]{Definition}
\theoremstyle{remark}
\newtheorem{rmk}[thm]{Remark}
\theoremstyle{remark}
\newtheorem{exm}[thm]{Example}
\newtheorem{assu}[thm]{Assumption}

\newcommand{\A}{\mathbb{A}}
\newcommand{\B}{\mathbb{B}}
\newcommand{\C}{\mathbb{C}}
\newcommand{\F}{\mathbb{F}}
\newcommand{\G}{\mathbb{G}}

\newcommand{\LL}{\mathbb{L}}

\newcommand{\Q}{\mathbb{Q}}
\newcommand{\RR}{\mathbb{R}}
\newcommand{\R}{\mathbb{R}}
\newcommand{\T}{\mathbb{T}}
\newcommand{\Z}{\mathbb{Z}}

\newcommand{\cat}{\mathbf{C}}
\newcommand{\catD}{\mathbf{D}}

\newcommand{\catT}{\mathbf{T}}

\newcommand{\ra}{\rightarrow}
\newcommand{\mcB}{\mathcal{B}}

\newcommand{\mcF}{\mathcal{F}}

\newcommand{\mcO}{\mathcal{O}}

\newcommand{\mcU}{\mathcal{U}}

\newcommand{\mfB}{\mathfrak{B}}

\newcommand{\mfG}{\mathfrak{G}}

\newcommand{\mfT}{\mathfrak{T}}

\newcommand{\mfX}{\mathfrak{X}}
\newcommand{\mfY}{\mathfrak{Y}}

\newcommand{\adj}[4]{#1\negmedspace: #2\rightleftarrows #3:\negmedspace #4}

\DeclareMathOperator{\an}{an}

\DeclareMathOperator{\Aff}{Aff}

\DeclareMathOperator{\br}{br}

\DeclareMathOperator{\car}{char}

\DeclareMathOperator{\cp}{cp}

\DeclareMathOperator{\dR}{dR}
\DeclareMathOperator{\eff}{eff}
\DeclareMathOperator{\et}{\acute{e}t}

\DeclareMathOperator{\Et}{Et}

\DeclareMathOperator{\FormDA}{\bf{FormDA}}
\DeclareMathOperator{\FormDM}{\bf{FormDM}}
\DeclareMathOperator{\Frac}{Frac}
\DeclareMathOperator{\Frob}{Frob}
\DeclareMathOperator{\Frobet}{Frob\acute{e}t}

\DeclareMathOperator{\Gal}{Gal}

\DeclareMathOperator{\gr}{gr}

\DeclareMathOperator{\Hom}{Hom}

\DeclareMathOperator{\id}{id}

\DeclareMathOperator{\perf}{{perf}}
\DeclareMathOperator{\Perf}{Perf}
\DeclareMathOperator{\PerfSm}{PerfSm}
\DeclareMathOperator{\sPerfSm}{sPerfSm}

\DeclareMathOperator{\rig}{rig}
\DeclareMathOperator{\Rig}{Rig}
\DeclareMathOperator{\RigCor}{RigCor}

\DeclareMathOperator{\RigSm}{RigSm}

\DeclareMathOperator{\Sing}{Sing}
\DeclareMathOperator{\Sm}{Sm}

\DeclareMathOperator{\Spa}{Spa}
\DeclareMathOperator{\Spec}{Spec}

\DeclareMathOperator{\Spf}{Spf}

\DeclareMathOperator{\tr}{tr}

\DeclareMathOperator{\Ch}{\bf{Ch}}

\DeclareMathOperator{\DA}{\bf{DA}}
\DeclareMathOperator{\DM}{\bf{DM}}

\DeclareMathOperator{\FormSm}{{FormSm}}

\DeclareMathOperator{\PerfDA}{\bf{PerfDA}}
\DeclareMathOperator{\sPerfDA}{\bf{sPerfDA}}
\DeclareMathOperator{\PerfDM}{\bf{PerfDM}}

\DeclareMathOperator{\RigDA}{\bf{RigDA}}
\DeclareMathOperator{\RigDM}{\bf{RigDM}}

\DeclareMathOperator{\Psh}{\bf{Psh}}
\DeclareMathOperator{\PST}{\bf{PST}}

\DeclareMathOperator{\Sh}{\bf{Sh}}

\DeclareMathOperator{\Un}{{\bf U}^n\!}
\DeclareMathOperator{\UU}{{\bf U}^1\!}

\begin{document}
\title{Rigid cohomology via the tilting equivalence} 
\author{Alberto Vezzani}
\address{LAGA - Universit\'e Paris 13, Sorbonne Paris Cit\'e, 
	99 av. Jean-Baptiste Cl\'ement, 93430 Villetaneuse, France}

\email{vezzani@math.univ-paris13.fr}

\begin{abstract}
We define a de Rham cohomology theory for analytic varieties over a valued field $K^\flat$ of equal characteristic $p$ with coefficients in a chosen untilt of the perfection of $K^\flat$ by means of the motivic version of Scholze's tilting equivalence. We show that this definition generalizes the usual rigid cohomology in case the variety has good reduction. We also prove a conjecture of Ayoub yielding an equivalence between rigid analytic motives with good reduction and unipotent algebraic motives over the residue field, also in mixed characteristic.
\end{abstract}

\maketitle

\tableofcontents

\section{Introduction}

Rigid cohomology can be considered as a substitute of the pathological (at least for non-proper varieties) de Rham cohomology over a field $k$ of positive characteristic. Its definition is based on the idea of associating to a variety $\bar{X}/k$  another variety $X$ over a field $K$ of characteristic $0$ and then considering its well-behaved de Rham cohomology.

The classic path to reach this goal, elaborating on the work of Monsky and Washnitzer, is to find a smooth formal model $\mfX$ of $\bar{X}$ over a valuation ring $\mcO_K$ of mixed characteristic, then to consider its (rigid analytic) generic fiber $X$. % 

From Fontaine's and Scholze's work in $p$-adic Hodge theory, we are now accustomed with another strategy to change characteristic: the so-called \emph{tilting equivalence}. This equivalence is built between perfectoid spaces (a certain kind of adic spaces, see \cite{scholze}) over a perfect,  non-archimedean field $K^\flat$  of characteristic $p$, and perfectoid spaces over a fixed \emph{untilt} $K$ of $K^\flat$. We recall that such \emph{untilts} are perfectoid fields (that is, complete with respect to a non-discrete valuation of rank $1$, with  residue characteristic equal to $p>0$ and such that the Frobenius map is surjective over $\mcO_K/p$)  of mixed characteristic, have the same absolute Galois group as $K^\flat$ and are parametrized by principal ideals of $W(\mcO_{K^\flat})$ generated by a primitive element of degree one (see \cite{fontaine-bourbaki} and  \cite{scholze}). Needless to say, this perfectoid framework of Scholze has been reaping many interesting results in $p$-adic Hodge theory, and in arithmetic geometry in general (see \cite{scholze-icm}).

An alternative method for making $\bar{X}$  ``change characteristic'' is then at reach: we can base change  $\bar{X}$ to a perfectoid field $K^\flat$ then we can take its analytification, followed by its perfection. This defines a perfectoid space $\widehat{X}^\flat$ over $K^\flat$ that we can finally switch (by Scholze's equivalence) to a perfectoid space $\widehat{X}$ over a chosen untilt $K$ of $K^\flat$ having characteristic $0$. 

The initial problem, namely the definition of a de Rham cohomology for $\bar{X}$, is not yet solved with the above procedures. Indeed, for  rigid analytic varieties as well as for perfectoid spaces, the de Rham complex is still problematic (its cohomology groups can be oddly infinite-dimensional for smooth, affinoid rigid analyitic varieties).  
Nonetheless, the results of \cite{gk-dR} and \cite{vezz-fw}  show that   some natural de Rham cohomology groups for both smooth rigid analytic varieties as well as smooth perfectoid spaces can truly be defined.

One aim of this paper is that the two recipes mentioned above are actually equivalent. We now rephrase this statement in terms of motives. This allows for a more precise result, and further corollaries that we examine afterwards.

Out of the categories of smooth varieties over a field, smooth rigid analytic varieties over a non-archimedean field, smooth perfectoid spaces over a perfectoid field, or smooth formal schemes over a valuation ring, one can construct the associated category of motives, written as $\DM(K)$ (or $\RigDM(K)$, $\PerfDM(K)$ and $\FormDM(\mcO_K)$ in the various cases). We consider here derived, effective motives with rational coefficients following Voevodsky. In particular, these categories are Verdier quotients of the classical derived categories of \'etale sheaves with $\Q$-coefficients defined on each big site.

The special fiber and the generic fiber functors on varieties induce some functors also at the level of motives:
\[
\xi\colon \DM(k)\stackrel{\sim}{\leftarrow}\FormDM(\mcO_K)\ra\RigDM(K)
\]
(the first one has a natural quasi-inverse as proven in \cite{ayoub-rig}). The composition from left to right is then the motivic version of the first recipe (Monsky-Washnitzer's) sketched above.

Suppose now $K$ is perfectoid. It is possible to rephrase Scholze's tilting equivalence in motivic terms by saying that the categories $\PerfDM(K)$ and $\PerfDM(K^\flat)$ are equivalent, where $K^\flat$ is its unique \emph{tilt} in characteristic $p$ (see \cite{scholze}). In \cite{vezz-fw} we descended this result to the rigid analytic situation, by proving the following:

\begin{thm*}[{\cite{vezz-fw}}]Let $K$ be a perfectoid field. 
	There is a canonical monoidal, triangulated equivalence 
	$$\mfG 
	\colon
	 \RigDM(K)\cong\PerfDM(K).$$
	In particular, for any perfectoid field $K$ of mixed characteristic, we obtain a canonical monoidal, triangulated equivalence 
	$$
	\mfG_{K,K^\flat}\colon \RigDM(K)\cong\RigDM(K^\flat).
	$$
\end{thm*}

We can then consider the following functors
\[
\xi^\flat\colon\DM(k)\ra\DM(K^\flat)\stackrel{\Rig^*}{\ra}\RigDM(K^\flat)\cong\RigDM(K)
\]
and their composition corresponds to the second recipe that we sketched above. We then prove (see Theorem \ref{main} and Proposition \ref{chiflat}):
\begin{thm*}
Let $K$ be a perfectoid field. The following diagram commutes, up to an invertible natural transformation.
		\[
		\xymatrix{
			&\DM(k)\ar[dl]_{\xi}\ar[dr]^{\xi^\flat}\\
			\RigDM(K) %\ar[ur]^{\chi}
			\ar@{<->}[rr]^{\sim}_{\mfG_{K,K^\flat}}&&\RigDM(K^\flat) 
		}
		\] 
\end{thm*}
This tells us in particular that the method based on Scholze's tilting functor
 is just as good as Monsky-Washnitzer's (and Berthelot's) for defining rigid cohomology, up to enlarging the coefficient ring to a perfectoid field $K$ (see Corollary \ref{rigtilt}).

Incidentally, we point out that this second method is directly generalizable to (algebraic, perfectoid or) analytic varieties over $K^{\flat}$ giving rise to a cohomology theory ``\`a la de Rham'' that takes values in $K$-vector spaces (see Section \ref{secdR}). To our knowledge, this is the first time that such a definition appears in the literature, the ring of coefficients  $K$ being ``smaller'' than the one used for rigid cohomology over $K^\flat$ and different from the fields of coefficients considered by Lazda and Pal \cite{lp} (but relations between these approaches are envisaged). This definition can be applied in particular to families $\{\bar{X}_t\}$ of varieties over a field of characteristic $p$ which are generically smooth, and coincides with rigid cohomology of the special fiber $\bar{X}_0$ in case the family is smooth everywhere.  We therefore obtain a generalization of the rigid cohomology of a log-scheme (in the case of  semi-stable reduction) without the need of introducing a log-structure.

The technical difficulty in the proof is the construction of a natural transformation between the two functors. Due to the intricate definition of $\mfG$ the way we obtain it is indirect, via the introduction of some auxiliary categories (of semi-perfectoid spaces over $\mcO_K$ and over $k$) and the generalization of some results in \cite{vezz-fw}. Once the transformation is defined, one can deduce it is invertible from some explicit computations on very special motives generating $\DM(k)$ (see Theorem \ref{premain}).

Our main theorem has further, non-obvious consequences that are not necessarily correlated with rigid cohomology. We briefly describe some of them.
\begin{enumerate}
\item In general, there is not an explicit, geometric formula giving the image of the motive of a variety $X$ by the motivic tilting equivalence $\mfG_{K,K^\flat}\colon \RigDM(K)\cong\RigDM(K^\flat)$. From the theorem above, we find such a description for varieties $X$ with good reduction, i.e. having a smooth model $\mfX$ over $\mcO_K$. Their tilt is simply the motive of (the base change of the analytification of) their reduction $\bar{X}$ i.e. the special fiber $\mfX_\sigma$ (see Proposition \ref{tiltgr}).
\item Ayoub \cite{ayoub-rig} has proved that in the equal-characteristic zero case, the category $\RigDM(K)$ has a convenient description as a full, triangulated subcategory of motives above $\G_{m,k}$. The same is true for the equal-characteristic $p$ case, at the cost of considering only motives with good reduction, and introducing some technical hypotheses on the field. The mixed-characteristic counterpart was conjectured in \cite{ayoub-rig} and we are able to prove it here (see Theorem \ref{ayoubgr+}).
\item The weak tannakian formalism  \cite{ayoub-h1}, \cite{ayoub-h2} allows to use the Betti realization functor $\DM(K)\ra\catD(\Q)$ to define a motivic Galois group for subfields $K$ of $\C$. Moreover, the equivalence of the previous point  also gives rise to a Galois group  for a non-archimedean field of equi-characteristic zero. This induces relations between the absolute motivic Galois group of a field,  the motivic Galois group of its completion over a non-archimedean place of equal characteristic $0$ and the motivic Galois group of the residue field, mimicking  the relations between the absolute Galois groups and the decomposition group associated to a non-archimedean place. We can use the de Rham realization to obtain similar statements in any characteristic. We plan to study this in detail in a later work.
\end{enumerate} 
 
 We warn the reader that we mostly use here the variant of motives \emph{without transfers}, denoted typically with $\DA$ rather than $\DM$. These two approaches are canonically equivalent whenever the base field is perfect and the ring of coefficients contains $\Q$ (see \cite{vezz-DADM}): the one without transfers has the quality of being more direct for definitions.
 
  This article is organized as follows. In Section \ref{secsp} we introduce some accessory geometric categories over a valued field $K$, its ring of valuation $\mcO_K$ as well as its residue field $k$ and we define the associated categories of motives. In Section \ref{secmain} we prove our main theorem. In Section \ref{secdR} we describe how to define a de Rham cohomology over a field of equi-characteristic $p$ while in Section \ref{secconj} we solve Ayoub's conjecture on rigid analytic motives of good reduction in mixed characteristic.

\section{Motives of semi-perfectoid spaces over $K^\circ$ and $k$}
\label{secsp}

We briefly recall (see \cite{scholze}) that a \emph{perfectoid field} is a complete field with respect to a non-discrete valuation of rank $1$, with  residue characteristic equal to $p>0$ and such that the Frobenius map $\Frob$ is surjective over $\mcO_K/p$. For example,  the completion of $\Q_p(\mu_{p^\infty})$ and  $\C_p$ are perfectoid. Such an object can be associated canonically to a perfectoid field of characteristic $p$, which is called its \emph{tilt}, obtained as the fraction field of $\varprojlim_{\Frob}\mcO_K/p$. For example, the tilt of the completion of $\Q_p(\mu_{p^\infty})$ is the completion of $\F_p(\!(t^{1/p^\infty})\!)$ and the tilt of $\C_p$ is the completion of an algebraic closure of $\F_p(\!(t)\!)$.

In this section we make the following hypotheses. 
\begin{assu}\label{assKisperf}
We let $K$ be a perfectoid field,  we denote its tilt by $K^\flat$  and we fix an invertible, topologically nilpotent element $\pi\in K$ (a so-called \emph{pseudo-uniformizer} of $K$).  The valuation ring $\mcO_K$ will be  denoted by $K^\circ$ and the residue field by $k$.  
\end{assu}

For the notions of perfectoid spaces over perfectoid fields and their properties we refer to \cite{scholze}. We will make constant use of Huber's notations \cite{huber} for rigid analytic spaces and their generalizations. In particular, we will use the notation $\Spa(R,R^\circ)$ (which we will typically abbreviate by $\Spa R$) to indicate the analytic spectrum associated to a (well-behaved) topological ring $R$ and its maximal integral structure, given by the subring of power-bounded elements $R^\circ$. Also, we will use the symbol $K\langle\nu\rangle$ (and similar variants) to denote the  $\pi$-adic completion of the ring of polynomials $K[\nu]$.
 
 We fix here some notations and examples. 

\begin{exm}
\begin{enumerate}
\item The space $\widehat{\B}^1\colonequals\Spa(K\langle \upsilon^{1/p^\infty}\rangle)$ is perfectoid. The global function $\upsilon$ induces a map $\widehat{\B}^1\ra\B^1$ where $\B^1 =\Spa K\langle \upsilon\rangle$ is the usual closed rigid disc. This map also induces a map of formal schemes $\widehat{\mfB}^1\colonequals\Spf K^\circ\langle \upsilon^{1/p^\infty}\rangle\ra\mfB^1\colonequals \Spf K^\circ\langle \upsilon\rangle$ with special fiber $\widehat{\A}^1\colonequals\Spec k[ \upsilon^{1/p^\infty}]\ra\A^1= \Spec k[\upsilon]$.
\item The maps above restrict to the rational subspaces defined by the equation $|\upsilon|=1$ that is, we have maps $\widehat{\T}^1\colonequals\Spa K\langle \upsilon^{\pm1/p^\infty}\rangle\ra\T^1\colonequals\Spa K\langle \upsilon^{\pm1}\rangle$ and $\mfT^1\colonequals\Spf K^\circ\langle \upsilon^{\pm1/p^\infty}\rangle\ra\mfT^1\colonequals\Spf K^\circ\langle \upsilon^{\pm1}\rangle$ with special fiber $\widehat{\G}_m^1\colonequals\Spec k[ \upsilon^{\pm1/p^\infty}]\ra\G_m= \Spec k[\upsilon^{\pm1}]$.
\item For any positive integer $N$ we denote by $X^N$ the $N$-fold product of $X$ therefore defining the spaces $\T^N$, $\widehat{\T}^N$, $\mfT^N$ etc.
\end{enumerate}
\end{exm}

In \cite{vezz-fw} we defined the category of semi-perfectoid spaces over $K$. These objects form a convenient category of  adic spaces containing both smooth rigid analytic spaces and  smooth perfectoid spaces. A crucial property they have is that they are \emph{sheafy} (see \cite{bv}) which means that the structural presheaf $\mcO_X$ is actually a sheaf of complete topological rings for the analytic topology. We recall briefly their definition.

\begin{dfn}\label{defsemip}
The category of \emph{smooth semi-perfectoid spaces over $K$} is the full subcategory of adic spaces formed by spaces which are  locally isomorphic to $X_0\times_{\T^N}\widehat{\T}^N$ for some  \'etale morphism of rigid analytic varieties $X_0\ra\T^N\times\T^M$. We remark that $X_0\times_{\T^N}\widehat{\T}^N\sim \varprojlim X_h$ where $X_h\colonequals X_0\times_{\T^N,\varphi_h}{\T}^N$ and the map $\varphi_h$ is induced by raising the coordinates $\upsilon$ to the $p^h$-th power (for the precise meaning of $\sim$ we refer to \cite[Section 2.4]{huber}).

If one imposes $N=0$ in the definition above, one recovers the usual full subcategory category $\RigSm/K$ of smooth rigid analytic varieties over $K$ while if one imposes $M=0$, one defines the full subcategory $\PerfSm/K$ of \emph{smooth perfectoid spaces }over $K$.
\end{dfn}

As remarked in \cite{vezz-fw}, for all the categories above we can define the \'etale topology by considering Scholze's definition  (see \cite[Definition 7.1]{scholze}).

We now make the analogous definition in the case of smooth formal schemes over $K^\circ$ and of smooth algebraic varieties over $k$. Such geometrical categories will only play an accessory role in our main theorem. On the other hand, we remark that the category $\PerfSm/K^\circ$ is at the core of the constructions made in \cite{bms}.

\begin{dfn}\label{defsemip2}
A \emph{smooth semi-perfectoid space over $k$} [resp. $K^\circ$] is a [formal] scheme which is Zariski locally isomorphic to $X_0\times_{\G_m^N}\widehat{\G}_m^N$ [resp. $\mfX_0\times_{\mfT^N}\widehat{\mfT}^N$] for some [formally] \'etale morphism $X_0\ra\G_m^N\times\G_m^M$ [resp. $\mfX_0\ra\mfT^N\times\mfT^M$]. The full subcategory of schemes they form will be denoted by $\sPerfSm/k$ [resp. $\sPerfSm/K^\circ$]. Locally, a smooth semi-perfectoid space is the inverse limit of the system $X_h\colonequals X_0\times_{\G_m^N,\varphi_h}\G_m^N $ [resp. $\mfX_h\colonequals \mfX_0\times_{\mfT_m^N,\varphi_h}\mfT_m^N$] where $\varphi_h$ is induced by raising the coordinates $\upsilon$  to the $p^h$-th power.

If one imposes $N=0$ in the definitions above, one recovers the usual full subcategory category $\Sm/k$ of smooth schemes over $k$ [resp. the category $\FormSm/K^\circ$  of smooth formal schemes topologically of finite type over $K^\circ$]. If one imposes $M=0$ in the definitions above, one defines the full subcategory $\PerfSm/k$ of \emph{smooth perfectoid spaces over $k$ }[resp. $\PerfSm/K^\circ$ of \emph{smooth formal perfectoid spaces over $K^\circ$}].
\end{dfn}

\begin{rmk}
	The terminology ``smooth [semi-]perfectoid space'' is not standard, and here means ``locally \'etale over the perfectoid poly-disc'', matching the notion of smoothness for rigid analytic varieties (see for example \cite[Corollary 1.1.51]{ayoub-rig}) . On the other hand,  smooth semi-perfectoid spaces over  $K^\circ$ are sometimes called in the literature ``locally small'' (see for example \cite[Remark 1.15]{bms}) following Faltings's terminology \cite{faltings-pS}.
\end{rmk}

For all the categories above, we can define the \'etale topology by specializing the definition of \'etale maps and coverings for schemes and formal schemes. %

\begin{prop}\label{spgenfib}
	Let $\Spf A$ be \'etale over $\widehat{\mfT}^N\times\mfT^M$. The generic fiber functor $(-)_\eta$ induced by $\Spf A\mapsto\Spa(A[1/\pi],A)$ and the special fiber functor $(-)_\sigma$ induced by $\Spf A\mapsto\Spec(A\otimes k)$ define well-defined functors  to the \'etale  topos of $\sPerfSm/K^\circ$ from the one of $\sPerfSm/K$ and of $\sPerfSm/k$ respectively. These functors restrict to functors  to the topos of  formal schemes of tft from the one of smooth rigid varieties over $K$ and of smooth schemes over $k$ respectively. They also restrict to  maps on the associated \'etale topoi of perfectoid spaces.
\end{prop}

\begin{proof}
All the adic spaces involved in the statement are sheafy, so the claim follows from \cite[Lemma 3.5.1(1)]{huber}.
\end{proof}

Out of these geometric categories, we can immediately define the associated \emph{motives}. These (infinity-)categories are the natural environment where to define and study homotopies in an algebraic geometrical setting, and they satisfy interesting universal properties related to Weil cohomology theories, given by their construction.

 We use here a triangulated, more down-to-earth definition rather the model-categorical one we used in \cite{vezz-fw} (recalled in Remark \ref{modcatmot} and that we will freely use in proofs). Motives will be here defined as quotients of triangulated categories, namely the derived categories of $\Lambda$-sheaves over a site. The quotient is taken in order to impose homotopy-invariance on the (co-)homological functors, i.e. invariance with respect to the projections  $X\times I\ra X$ induced by a chosen ``interval object" $I$ in the geometric category. %

\begin{dfn}
	Fix a commutative ring $\Lambda$.  	Let  $\kappa$ be in $\{k,K^\circ,K\}$ and fix an object $I$ in $\sPerfSm/\kappa$. 
	\begin{itemize}
	\item 	We denote by $\Lambda(X)$ the presheaf with values in $\Lambda$-modules represented by $X$ i.e. the presheaf associating $Y$ to the free $\Lambda$-module associated to the set $\Hom(Y,X)$. 
	\item 	We denote by  $\Frobet$  the topology on $\sPerfSm/\kappa$ generated by \'etale covers and relative Frobenius morphisms (in case the base ring has positive characteristic) and let $a_{\Frobet}$ be  the associated sheafification functor. 
	\item The category of\emph{ effective semi-perfectoid motives} $\sPerfDA^{\eff}_{\Frobet,I}(\kappa,\Lambda)$ is the Verdier quotient of the derived category  of $\Frobet$-sheaves on smooth semi-perfectoid spaces with values in $\Lambda$-modules $\catD(\Sh_{\Frobet}(\sPerfSm/\kappa,\Lambda))$ over its sub-triangulated category with small sums generated by the cones of the projection maps $a_{\Frobet}\Lambda(X\times I)\ra a_{\Frobet}\Lambda(X)$ by letting $X$ vary in $\sPerfSm/\kappa$.% 
	\item We  make an analogous definition for smooth perfectoid spaces over $\kappa$ in which case the categories of motives will be denoted by $\PerfDA^{\eff}_{\Frobet,I}(\kappa,\Lambda)$. 
	\item We can also make the analogous definitions for smooth schemes over $k$, smooth formal schemes of tft over $K^\circ$ and smooth rigid varieties over $K$ and the associated motives are denoted by $\DA^{\eff}_{\Frobet,I}(k,\Lambda)$, $\FormDA^{\eff}_{\Frobet,I}(K^\circ,\Lambda)$ and $\RigDA^{\eff}_{\Frobet,I}(K,\Lambda)$ respectively.
	\end{itemize}
\end{dfn}

\begin{rmk}\label{modcatmot}
Fix a site $(\cat,\tau)$ and an object $I\in \cat$ (for example, we can take the $\Frob$-\'etale site on smooth rigid analytic varieties over $K$ and $I$ to be the ball $\B^1$). The category of the associated motives (in the example, the category $\RigDA^{\eff}(K,\Lambda)$) admits the following presentation as a homotopy category of a model category structure (see \cite[Chapter 4]{ayoub-th2}). 

One can first consider the projective model structure on (unbounded) complexes of presheaves on $\cat$ with coefficients in $\Lambda$. It contains representable presheaves $\Lambda(X)$ as well as the complexes $\Lambda(\mcU)$  associated to any simplicial object $\mcU$ of $\cat$. Its homotopy category is the (unbounded) derived category $\catD(\Psh(\cat,\Lambda))$.   Let $S$ be the class of maps containing all shifts of the the canonical arrows $\Lambda(X\times I)\ra\Lambda(X)$, $\Lambda(\mcU)\ra\Lambda(X)$ for any object $X$ and any $\tau$-hypercover $\mcU$ of it. The left Bousfield localisation of the projective model structure over $S$ is well defined, and its homotopy category coincides with the triangulated category of effective motives over the datum $(\cat,\tau,I)$. One can alternatively consider the category of (unbounded) complexes of \emph{sheaves} on $\cat$ and localize it with respect to the class of maps  $\Lambda(X\times I)\ra\Lambda(X)$ and their shifts. The interested reader can find all the details in \cite{ayoub-th2} (briefly resumed in \cite{vezz-fw}).

This presentation of motives as homotopy categories of a Quillen model structure equip them with a natural structure of a tensor DG-category, not merely of a a triangulated category. We remark that the tensor product $\Lambda(X)\otimes\Lambda(Y)$ is isomorphic to $\Lambda
(X\times Y)$. The language of model categories is also very convenient to define natural adjoint functors of motivic categories out of functors at the level of the underlying sites $(\cat,\tau)$.
\end{rmk}

\begin{exm}
We now list the objects $I$ that we will consider. 
\begin{enumerate}
\item The disc of radius one $\B^1=\Spa K\langle\upsilon\rangle$ is an object of $\RigSm/K$. It has a smooth formal model $\mfB^1=\Spf K^\circ\langle \upsilon\rangle$ which is an object of $\FormSm/K^\circ$. Its special fiber is the affine line $\A^1_k$ in $\Sm/k$.
\item The perfectoid disc of radius one $\widehat{\B}^1=\Spa K\langle\upsilon^{1/p^\infty}\rangle$ is an object of $\PerfSm/K$. It has a smooth formal model $\widehat{\mfB}^1=\Spf K^\circ\langle \upsilon^{1/p^\infty}\rangle$ which is an object of $\PerfSm/K^\circ$ since it is covered by the two subspaces $\Spf K^\circ\langle \upsilon^{\pm1/p^\infty}\rangle$ and $\Spf K^\circ\langle (\upsilon+1)^{\pm1/p^\infty}\rangle$ both isomorphic to $\widehat{\mfT}^1$. Its special fiber is $\widehat{\A}^1_k\colonequals \Spec k[\upsilon^{1/p^\infty}]$ which lies in $\PerfSm/k$.
\end{enumerate} 
\end{exm}

In order to have simpler notations, from now on we will adopt some alternative conventions for the   categories of motives that we will use the most.
\begin{dfn}We make the following abbreviations.
	\begin{itemize}
\item Whenever the interval object is the ``natural one'' we will drop it from the notation, as follows:\\	$\DA^{\eff}_{\Frobet}(k,\Lambda)\colonequals\DA^{\eff}_{\Frobet,\A^1}(k,\Lambda)$,\\	$\FormDA^{\eff}_{\Frobet}(K^\circ,\Lambda)\colonequals\FormDA^{\eff}_{\Frobet,\mfB^1}(K^\circ,\Lambda)$, \\ $\RigDA^{\eff}_{\Frobet}(K,\Lambda)\colonequals\RigDA^{\eff}_{\Frobet,\B^1}(K,\Lambda)$.
\item Moreover, whenever the \'etale topology is considered, we sometimes omit it from the notation and we then put:\\
$\DA^{\eff}(k,\Lambda)\colonequals\DA^{\eff}_{\et,\A^1}(k,\Lambda)$,\\
$\FormDA^{\eff}(K^\circ,\Lambda)\colonequals\FormDA^{\eff}_{\et,\mfB^1}(K^\circ,\Lambda)$, \\
 $\RigDA^{\eff}(K,\Lambda)\colonequals\RigDA^{\eff}_{\et,\B^1}(K,\Lambda)$,\\
  $\sPerfDA^{\eff}_{I}(\kappa,\Lambda)\colonequals\sPerfDA^{\eff}_{\et,I}(\kappa,\Lambda)$,\\
  $\sPerfDA^{\eff}_{\widehat{I}}(\kappa,\Lambda)\colonequals\sPerfDA^{\eff}_{\et,\widehat{I}}(\kappa,\Lambda)$, \\
  $\PerfDA^{\eff}(\kappa,\Lambda)\colonequals\PerfDA^{\eff}_{\et,\widehat{I}}(\kappa,\Lambda)$
  \\ where we let  $\kappa$ be in $\{k,K^\circ,K\}$ and $I$ [resp. $\widehat{I}$] be ${\A}^1_k$, ${\mfB}^1$ or ${\B}^1$ [resp. $\widehat{\A}^1_k$, $\widehat{\mfB}^1$ or $\widehat{\B}^1$] accordingly.  
\item We will assume $\Q\subset\Lambda$ and we drop $\Lambda$ from the notation, if the context allows it.
	\end{itemize}
\end{dfn}

\begin{rmk}
Suppose that $\car K=p$. If one considers the \'etale topology instead of the finer $\Frobet$ topology, one gets the category of \'etale motives $\RigDA^{\eff}_{\et}(K,\Lambda)$ and a natural functor $\RigDA^{\eff}_{\et}(K,\Lambda)\ra\RigDA^{\eff}_{\Frobet}(K,\Lambda)$. This functor coincides with the Verdier quotient over the subcategory generated by cones of the relative Frobenius maps $X\ra X\times_{\Frob}K$ (see \cite{vezz-DADM}).
\end{rmk}
\begin{rmk}
	For semi-perfectoid motives, we remark that the localizations over the perfectoid disc (or the perfect affine line) factor over the localizations with respect to the disc (or the affine line). That is: homotopy equivalences with respect to the disc are automatically homotopy equivalences with respect to the {perfectoid} disc (see for example  \cite[Section 3]{vezz-fw}). In particular, there is  an adjunction (more precisely, a Bousfield localisation) $$\sPerfDA^{\eff}_{I}(\kappa,\Lambda)\rightleftarrows\sPerfDA^{\eff}_{\widehat{I}}(\kappa,\Lambda)$$
	where  we let  $\kappa$ be in $\{k,K^\circ,K\}$ and $I$ [resp. $\widehat{I}$] be ${\A}^1_k$, ${\mfB}^1$ or ${\B}^1$ [resp. $\widehat{\A}^1_k$, $\widehat{\mfB}^1$ or $\widehat{\B}^1$] accordingly.
\end{rmk}

The knowledgeable reader will notice that we use here the version of (effective, derived) motives \emph{without transfers}  rather than the version  \emph{with transfers}, typically denoted with $\DM^{\eff}$ and mentioned in the introduction. We recall here briefly their definition, which makes use of correspondences. For more properties of such categories, we refer to \cite[Chapter 2]{ayoub-rig}.

\begin{dfn} Ler $K'$ be any complete valued field (possibly trivially valued).
	\begin{itemize}
		\item We define the category $\RigCor/K'$ as the category whose objects are those of $\RigSm/K'$ and whose morphisms $\Hom(X,Y)$ are the free $\Lambda$ modules generated by the set of integral, closed subvarieties $Z$ of $X\times Y$ which are finite and surjective over a connected component of $X$ and where composition is induced by the intersection formula (see \cite[Remark 2.2.21]{ayoub-rig}).
\item 		The category $\Psh(\RigCor/K')$ will be denoted by $\PST(\RigSm/K')$. Its full subcategory of those presheaves $\mcF$ that are sheaves when restricted to $\RigSm/K'$ will be denoted by $\Sh_{\tr}(\RigSm/K')$. The sheaf represented by a smooth variety $X$ is denoted by $\Lambda_{\tr}(X)$.
	\item The category of effective rigid analytic motives with transfers $\RigDM^{\eff}(K',\Lambda)$ is the Verdier quotient of the derived category of $\Sh_{\tr}(\RigSm/K')$ over its sub-triangulated category with small sums generated by the cones of the projection maps $\Lambda_{\tr}(X\times\B^1)\ra\Lambda_{\tr}(X)$ by letting $X$ vary in $\RigSm/K'$. In case $K'$ is trivially valued, it will simply be denoted with  $\DM^{\eff}(K',\Lambda)$
	\end{itemize}
\end{dfn}

\begin{rmk}\label{modcatmotDM}
Also the category  $\RigDM^{\eff}(K',\Lambda)$ admits a presentation as a homotopy category of a model category structure, analogous to the case of motives without transfers (Remark \ref{modcatmot}).
	
	One can first consider the projective model structure on  complexes $\Ch\PST(K',\Lambda)$.  Its homotopy category is the (unbounded) derived category $\catD(\Psh(\cat,\Lambda)$.   Let $S$ be the class of maps containing all shifts of the the canonical arrows $\Lambda_{\tr}(X\times I)\ra\Lambda_{\tr}(X)$, $\Lambda_{\tr}(\mcU)\ra\Lambda_{\tr}(X)$ for any smooth variety $X$ and any $\et$-hypercover $\mcU$ of it. The left Bousfield localisation of the projective model structure over $S$ is well defined, and its homotopy category coincides with $\RigDM^{\eff}(K',\Lambda)$. The interested reader can find all the details in \cite{ayoub-rig}.%
	
	This presentation of motives as homotopy categories of a Quillen model structure equip them with a natural structure of a tensor DG-category, not merely of a a triangulated category. We remark that the tensor product $\Lambda_{\tr}(X)\otimes\Lambda_{\tr}(Y)$ is isomorphic to $\Lambda_{\tr}(X\times Y)$.
\end{rmk}

Adding correspondences or not makes no difference in our context as soon as $\Q\subset\Lambda$ thanks to the following  fact (which is classical in a vast list of situations, see \cite[Appendix B]{ayoub-h1} and the other references of \cite{vezz-DADM}).

\begin{thm}[{\cite{vezz-DADM}}]\label{DADM}
	Under Assumption \ref{assKisperf} and the hypothesis that  $\Lambda$ is a $\Q$-algebra, there are canonical triangulated, monoidal (Quillen) equivalences
		$$
		\DA^{\eff}_{\Frobet}(k,\Lambda)\cong\DM_{\et}^{\eff}(k,\Lambda)\qquad
		\RigDA^{\eff}_{\Frobet}(K,\Lambda)\cong\RigDM_{\et}^{\eff}(K,\Lambda)
		$$
\end{thm}

\begin{rmk}
The canonical functor appearing in the theorem above is the derived functor $\LL a_{\tr}$ induced by the canonical inclusion $\Sm/K\ra\RigCor/K$.
\end{rmk}

We now list the adjoint pairs that link the various categories of motives we introduced above. They are all defined thanks to the \emph{functoriality} of the very definition of motives, in the following sense:  motives are defined as appropriate Verdier quotients of the derived category $\catD(\Psh(\cat,\Lambda))$ of $\Lambda$-presheaves on some category $\cat$. Any functor $F\colon\cat\ra\cat'$ induces a (Quillen) adjoint pair between the derived categories of the associates   $\Lambda$-presheaves 
$$\LL F^*\colon\catD\Psh(\cat,\Lambda)\leftrightarrows\catD\Psh(\cat',\Lambda)\colon\RR F_*.$$
This pair is such that $\LL F^*$ sends a representable presheaf $\Lambda(X)$ to the representable presheaf $\Lambda(F(X))$. If the functor $\LL F^*$ factors over the Verdier quotient, we then obtain a (left adjoint) functor defined on the motivic category.%

\begin{thm}\label{derij}
		\begin{enumerate}
\item		The  inclusions $\iota\colon\Sm/k\ra\sPerfSm/k$ resp. $\iota\colon\FormSm/K^\circ\ra\sPerfSm/K^\circ$ resp. $\iota\colon\RigSm/K\ra\sPerfSm/K$ induce the following (Quillen)  triangulated adjunctions, whose left adjoint is monoidal:
		$$
		\adj{\LL\iota^*}{\DA^{\eff}(k,\Lambda)}{\sPerfDA^{\eff}_{\A^1}(k,\Lambda)}{\RR\iota_*}
		$$
		$$
		\adj{\LL\iota^*}{\FormDA^{\eff}(K^\circ,\Lambda)}{\sPerfDA^{\eff}_{\mfB^1}(K^\circ,\Lambda)}{\RR\iota_*}
		$$
		$$
		\adj{\LL\iota^*}{\RigDA^{\eff}(K,\Lambda)}{\sPerfDA^{\eff}_{\B^1}(K,\Lambda)}{\RR\iota_*}
		$$
		\item Let $\kappa$ be in $\{k,K^\circ,K\}$. We let $\widehat{I}$ be $\widehat{\A}^1_k$, $\widehat{\mfB}^1$ or $\widehat{\B}^1$ accordingly. 
			The canonical inclusion $j\colon\PerfSm/\kappa\ra\sPerfSm/\kappa$ induces the following (Quillen)  triangulated adjunction, whose left adjoint is monoidal:
			$$
			\adj{\LL j^*}{\PerfDA^{\eff}(\kappa,\Lambda)}{\sPerfDA^{\eff}_{\widehat{I}}(\kappa,\Lambda)}{\RR j_*}
			$$
			\end{enumerate}
\end{thm}

\begin{proof}We use the definition of motives as homotopy categories given in Remark \ref{modcatmot}. The fact that the maps of topoi induce such Quillen adjunctions follows formally from the functoriality of the construction of motives, and the isomorphisms $\iota\A^1\cong\A^1$, $\iota{\mfB}^1\cong{\mfB}^1$, $\iota\mcB^1\cong\B^1$, $j\widehat{I}\cong\widehat{I}$. Monoidality follows from the formulas $\iota(X\times Y)\cong\iota X\times\iota Y$ and $j(X\times Y)\cong j X\times j Y$.%
\end{proof}

\begin{thm}\label{dd}
	%\hspace{2em}
	\begin{enumerate}
		\item	The maps of topoi $(-)_\sigma$ considered in Proposition \ref{spgenfib} induce the following (Quillen) monoidal, triangulated equivalences:
		$$
		\adj{\LL(\cdot)^*_\sigma}{\sPerfDA^{\eff}_{\hat{\mfB}^1}(K^\circ,\Lambda)}{\sPerfDA^{\eff}_{\hat{\A}^1}(k,\Lambda)}{\RR(\cdot)_{\sigma*}}
		$$
		$$
		\adj{\LL(\cdot)^*_\sigma}{\PerfDA^{\eff}(K^\circ,\Lambda)}{\PerfDA^{\eff}(k,\Lambda)}{\RR(\cdot)_{\sigma*}}
		$$
		$$
		\adj{\LL(\cdot)^*_\sigma}{\FormDA^{\eff}(K^\circ,\Lambda)}{\DA^{\eff}(k,\Lambda)}{\RR(\cdot)_{\sigma*}}
		$$
		\item
			The map of topoi $(-)_\eta$ induce the following (Quillen) triangulated adjunctions, whose left adjoint is monoidal:
			$$
			\adj{\LL(\cdot)^*_\eta}{\sPerfDA^{\eff}_{\mfB^1}(K^\circ,\Lambda)}{\sPerfDA^{\eff}_{\B^1}(K,\Lambda)}{\RR(\cdot)_{\eta*}}
			$$
			$$
			\adj{\LL(\cdot)^*_\eta}{\PerfDA^{\eff}(K^\circ,\Lambda)}{\PerfDA^{\eff}(K,\Lambda)}{\RR(\cdot)_{\eta*}}
			$$
			$$
			\adj{\LL(\cdot)^*_\eta}{\FormDA^{\eff}(K^\circ,\Lambda)}{\RigDA^{\eff}(K,\Lambda)}{\RR(\cdot)_{\eta*}}
			$$
	\end{enumerate}
\end{thm}

\begin{proof}
	We again use the definition of motives as homotopy categories given in Remark \ref{modcatmot}. 
The fact that the maps of topoi induce such Quillen adjunctions follows formally from the functoriality of the construction of motives, and the isomorphisms $(\mfB^1)_\eta\cong\B^1$, $(\widehat{\mfB}^1)_\eta\cong\widehat{\B}^1$, $(\mfB^1)_\sigma\cong\A^1$, $(\widehat{\mfB}^1)_\sigma\cong\widehat{\A}^1$. Monoidality follows formally from the formulas $(\mfX\times\mfY)_\eta\cong\mfX_\eta\times\mfY_\eta$ and  $(\mfX\times\mfY)_\sigma\cong\mfX_\sigma\times\mfY_\sigma$. We are left to prove that the each adjunction of the first set is actually an equivalence. 

To this aim, it suffices to adapt  the proof of \cite[Proposition 1.4.21(2)]{ayoub-rig} to the (semi-)perfectoid setting, and then use the argument given in \cite[Corollary 1.4.24]{ayoub-rig}. For the convenience of the reader, we recall here the main steps of the proof and, for simplicity, we only consider the setting of perfectoid spaces. 

First of all, we remark that motives $\Lambda(\bar{X})$ with $\bar{X}$ \'etale over $\widehat{\G}_m^N$ form a set of  generators for the category $\PerfDA^{\eff}(k,\Lambda)$ since any smooth perfectoid space is locally of this form. Any such $\bar{X}$ admits a model $\mfX$ which is \'etale over $\widehat{\mfT}^N$ by \cite[Theorem 18.1.2]{EGAIV4}. In particular, the functor $\LL(\cdot)_\sigma^*$ sends a class of  generators to a class of generators. We remark that these generators are also \emph{compact}, i.e. the covariant functor they represent commutes with small direct sums (see Definition \ref{cpt}), as shown in \cite[Proposition 3.18]{vezz-fw}.

By means of \cite[Lemma 1.3.32]{ayoub-rig}, we are left to prove that for such an $\mfX$ the natural map $\Lambda(\mfX)\ra\RR(\cdot)_{\sigma*}\LL(\cdot)^*_\sigma\Lambda(\mfX)$ is invertible. 
We now fix a smooth map $f\colon \mfX'\ra\mfX$ between smooth perfectoid spaces over $K^\circ$ (here, as usual, smooth means locally \'etale over a relative perfectoid $n$-torus) and we fix a section $s\colon \mfX_\sigma\ra\mfX'_\sigma$ of the  map $f_\sigma$ induced on the special fibers. We let $T_{\mfX',s}$ be the presheaf of sets on $\sPerfSm\!/\mfX$  associating to $\mfY$ the set of maps $g\in\Hom_{\mfX}(\mfY,\mfX')$ such that $g_\sigma$ factors over $s$. By means of \cite[Corollary 4.5.40]{ayoub-th2} it suffices to prove that $ T_{\mfX,s}\otimes\Lambda\ra\Lambda$ induces a weak equivalence in $\Ch(\Sh_{\et}(\Sm/\mfX),\Lambda)$ with respect to the $\widehat{\B}^1$-localisation of the usual ($\et$-local, projective) model structure on complexes of sheaves % 
  (see Remark \ref{modcatmot}).

We can make some further reductions by restricting to an \'etale covering $\{\mfX_i\}$ of $\mfX$ by means of the perfectoid analogue of \cite[Lemma 1.4.22]{ayoub-rig}. In particular, we can suppose that $\mfX'$ is \'etale over $\mfX''\colonequals \mfX\times\widehat{\mfB}^N$. We remark that the map $T_{\mfX',s}\ra T_{\mfX'',s''}$ (where $s''$ is the induced section) is invertible after \'etale sheafification: an explicit inverse is defined by associating to a section $g\colon \mfY\ra\mfX''$ the pullback $g\times_{\mfX''}\mfX'$ as shown in the second step of the proof of \cite[Proposition 1.4.21]{ayoub-rig}. We can then suppose that $\mfX'=\mfX\times\widehat{\mfB}^N$. Without loss of generality, we can also suppose that the section $s$ is the zero section. In this case, we can construct an explicit homotopy $\Hom(\mfY,\widehat{\mfB}^1)\times T_{\mfX',s}(\mfY)\ra T_{\mfX',s}(\mfY)$ between the identity of  $T_{\mfX',s}(\mfY)$ and the constant zero map, as follows:
$$
\mfY\ra\mfX\times\widehat{\mfB}^N\times\widehat{\mfB}\ra\mfX\times\widehat{\mfB}^N
$$
where the second arrow is induced by the multiplication $\widehat{\mfB}^N\times\widehat{\mfB}\ra\widehat{\mfB}^N$ (see the third step of the proof of \cite[Proposition 4.5.42]{ayoub-th2}).
\end{proof}

\begin{rmk}\label{xichi}
	
In particular, we obtain the following monoidal, triangulated left adjoint functors
$$
	\begin{aligned}
\xi\colonequals{\LL(\cdot)^*_\eta\RR(\cdot)_{\sigma*}}\colon&	\sPerfDA^{\eff}_{\hat{\A}^1}(k,\Lambda) {\longrightarrow}\sPerfDA^{\eff}_{\hat{\B}^1}(K,\Lambda) 
\\
\xi\colonequals{\LL(\cdot)^*_\eta\RR(\cdot)_{\sigma*}}\colon&	\PerfDA^{\eff}(k,\Lambda) {\longrightarrow}\PerfDA^{\eff}(K,\Lambda)
\\
\xi\colonequals{\LL(\cdot)^*_\eta\RR(\cdot)_{\sigma*}}\colon&	\DA^{\eff}(k,\Lambda) {\longrightarrow}\RigDA^{\eff}(K,\Lambda)
	\end{aligned}
	$$
whose right adjoints $\LL(\cdot)_\sigma^*\RR(\cdot)_{\eta*}$ will be typically denoted by $\chi$.
\end{rmk}

\begin{cor}
	Let $K$ be a perfectoid field of positive characteristic. If we fix a section $k\ra K^{\flat\circ}$, the functor $\RR(\cdot)_{\sigma*}$ is canonically isomorphic to $\LL(\cdot\widehat{\times}_{\Spec k}\Spf K^\circ)$. In particular, the motive $\Lambda(\mfX)$ is canonically isomorphic to $\Lambda(\bar{\mfX}\widehat{\times}_{\Spec k}\Spf K^\circ)$.
	\end{cor}
	
	\begin{proof}
Since $\RR(-)_{\sigma*}$ is an inverse of $\LL(-)_\sigma^*$ it suffices to prove that for any object $\bar{X}$ in $\sPerfSm/k$ the motive $\LL(-)_\sigma^*\Lambda(\bar{X}\widehat{\times}_{\Spec k}\Spf K^\circ)$ is isomorphic to $\Lambda(\bar{X})$ which is clear.
	\end{proof}

We complete the list of natural functors of motives with the ones induced by the analytification procedure   and the canonical ``perfection'' over a field of characteristic $p$.

\begin{thm}[{\cite[Proposition 1.4.14]{ayoub-rig}}]\label{anmot}
		The functor associating to an algebraic variety $X$ its analytification  $X^{\an}$ (see \cite[Section 1.1.3]{ayoub-rig}) induces the following (Quillen) monoidal, triangulated adjunction:
		$$
		\adj{\LL\Rig^*}{\DA^{\eff}(K,\Lambda)}{\RigDA^{\eff}(K,\Lambda)}{\RR\Rig_*}
		$$
\end{thm}

\begin{thm}[{\cite[Section 6]{vezz-fw}}]\label{LPerf}
	The functor associating to a  an algebraic variety $X$ over $k$ [resp. a rigid analytic variety $X$ over $K^\flat$] its [completed] perfection  $X^{\Perf}$  induces the following (Quillen) monoidal, triangulated adjunctions:
	$$
	\adj{\LL\Perf^*}{\DA^{\eff}_{\Frobet}(k,\Lambda)}{\PerfDA^{\eff}(k,\Lambda)}{\RR\Perf_*}
	$$
		$$
	\adj{\LL\Perf^*}{\RigDA^{\eff}_{\Frobet}(K^\flat,\Lambda)}{\PerfDA^{\eff}(K^\flat,\Lambda)}{\RR\Perf_*}
	$$
\end{thm}

\section{The commutativity}
\label{secmain}

In this section we still assume that $K$ is a perfectoid field with residue $k$ (see Assumpion \ref{assKisperf}). From now on, we also fix a section $k\ra K^{\flat\circ}$ and we assume that $\Q\subset\Lambda$. % 

We can rephrase Scholze's tilting equivalence between perfectoid spaces over $K$ and over $K^\flat$ \cite[Theorem 1.7]{scholze} in motivic terms by saying that the categories $\PerfDA^{\eff}(K,\Lambda)$ and $\PerfDA^{\eff}(K^\flat,\Lambda)$ are equivalent. In \cite{vezz-fw} we descended this result to the rigid analytic situation, by proving the following:

\begin{thm}[{\cite{vezz-fw}}]\label{mottilteq}
	Let $K$ be a perfectoid field. 
	There is a canonical monoidal, triangulated equivalence 
	$$\mfG= \RR j_*\circ\LL\iota^*\colon \RigDA^{\eff}_{\Frobet}(K,\Lambda)\cong\PerfDA^{\eff}(K,\Lambda).$$
	In particular, for any perfectoid field $K$ of mixed characteristic, we obtain a canonical monoidal, triangulated equivalence 
	$$
	\mfG_{K,K^\flat}\colon \RigDA^{\eff}(K,\Lambda)\cong\RigDA^{\eff}_{\Frobet}(K^\flat,\Lambda).
	$$
\end{thm}

Since the residue field $k$ of $K$ coincides with the one of $K^\flat$ \cite[Lemma 3.4]{scholze} we get by Remark \ref{xichi} two functors
$$\xi\colon\DA^{\eff}(k,\Lambda)\mathop{\longrightarrow}_{\sim}^{\RR(\cdot)_{\sigma*}}\FormDA^{\eff}(K^\circ,\Lambda)\stackrel{\LL(\cdot)_{\eta}^*}{\longrightarrow}\RigDA^{\eff}(K,\Lambda)$$
$$\xi^\flat\colon\DA^{\eff}(k,\Lambda)\mathop{\longrightarrow}_{\sim}^{\RR(\cdot)_{\sigma*}}\FormDA^{\eff}(K^{\flat\circ},\Lambda)\stackrel{\LL(\cdot)_{\eta}^*}{\longrightarrow}\RigDA^{\eff}(K^\flat,\Lambda)$$
whose right adjoints will be denoted by $\chi$ and $\chi^\flat$ respectively.

The aim of this section is to prove that they are compatible with the tilting equivalence. In other words, we want to prove the following result.

\begin{thm}\label{main}
	The following diagram is commutative, up to an invertible natural transformation.
		\[
				\xymatrix{
					&\DM^{\eff}(k,\Lambda)\ar[dl]_{\xi}\ar[dr]^{\xi^\flat}\\
					\RigDM^{\eff}(K,\Lambda) %\ar[ur]^{\chi}
					\ar@{<->}[rr]^{\sim}_{\mfG_{K,K^\flat}}&&\RigDM^{\eff}(K^\flat,\Lambda) %\ar[ul]_{\chi^{\flat}}
				}
		\]
\end{thm}

\begin{rmk}\label{chi}
The theorem is equivalent to the commutation of the adjoint diagram, whose sides are $\chi$ and $\chi^\flat$. It is also equivalent to the version with the categories without transfers (see Theorem \ref{DADM}). We will then consider the analogous diagram with $ \DA^{\eff}_{\Frobet}(k,\Lambda) $, $\RigDA^{\eff}(K,\Lambda)$ and $\RigDA^{\eff}_{\Frobet}(K^\flat,\Lambda) $ in the proof.
\end{rmk}

We  point out that the functor $\xi^\flat\colon\DA^{\eff}(k,\Lambda)\ra\RigDA^{\eff}(K^\flat,\Lambda)$ takes a particularly explicit form. Recall that we have chosen a section $k\ra K^{\flat\circ}$. 

\begin{dfn} [{\cite[Section 2.5]{ayoub-rig}}]
Let $\bar{X}=\Spec \bar{A}$ be a smooth affine scheme over $k$. We let $Q^{\rig}(\bar{X})$ be the rigid variety $\mfX_\eta$ where $\mfX$ is the formal scheme associated to the completion of $\bar{A}\otimes_k K^{\flat\circ}$.
\end{dfn}

\begin{prop}\label{chiflat}
	The functor $\xi^\flat\colon \DA^{\eff}(k,\Lambda)\ra\RigDA^{\eff}(K^\flat,\Lambda)$ is canonically isomorphic to the functor 
	  $\LL(\Rig\circ(\cdot\times K^\flat))^*$ and to the (Quillen) functor induced by $Q^{\rig}(\cdot)$.
\end{prop}

\begin{proof}
	The two last functors are canonically equivalent, as shown in {\cite[Lemma 2.5.59]{ayoub-rig}}.
	
	From the canonical isomorphism $(\bar{X}\times_kK^{\flat\circ})_\sigma\cong\bar{X}$ we deduce that $\LL(Q^{\rig}(\cdot))$ is a right inverse for $\LL(\cdot)^*_\sigma$ and therefore coincides canonically with $\RR(\cdot)_{\sigma*}$ by Theorem \ref{dd}, and we obtain the  claim. 
\end{proof}

The commutativity stated in Theorem \ref{main} will be reached by combining the two following propositions.

\begin{thm}\label{premain}
	Let $K$ be a perfectoid field. The functors $\iota$ and $j$ induce a commutative  diagram of left adjoint functors
$$\xymatrix{
	\DA^{\eff}_{\Frobet}(k,\Lambda)\ar[d]_{\xi}\ar[r]^-{\LL\iota^*}&  \sPerfDA^{\eff}_{\widehat{\A}^1}(k,\Lambda)\ar[d]_{\xi}	&	\PerfDA^{\eff}(k,\Lambda)\ar[l]_-{\LL j^*}\ar[d]_{\xi}\\
	\RigDA^{\eff}_{\Frobet}(K,\Lambda)	\ar[r]^-{\LL\iota^*}	& \sPerfDA^{\eff}_{\widehat{\B}^1}(K,\Lambda)	&	\PerfDA^{\eff}(K,\Lambda)\ar[l]_-{\LL j^*}\\	
	}
$$
which has the following extra properties:
\begin{enumerate}
\item \label{1} The compositions $\RR j_*\LL\iota^*$ defined on the two lines are monoidal, triangulated equivalences of categories. On the first line, the functor $\RR j_*\LL\iota^*$ is canonically isomorphic to $\LL\Perf^* $. If $\car K=p$ the functor $\mfG=\RR j_*\LL\iota^*$ on the second line is canonically isomorphic to $\LL\Perf^* $ as well.
\item There is a  natural transformation   
$\xi\circ\RR j_*\Rightarrow\RR j_*\circ\xi$.% 
\item The natural transformation $\xi\circ\LL\Perf^*\cong\xi\circ\RR j_*\circ\LL\iota^*\Rightarrow \RR j_*\circ\LL\iota^*\circ\xi$ is invertible.
\end{enumerate}
\end{thm}

\begin{proof}
We can reproduce the proofs of \cite[Proposition 6.6 and Theorem 6.9]{vezz-fw} for the case of a trivial valuation (in which case the proofs are simplified) in order to obtain the following facts:
	\begin{enumerate}[(i)]
		\item The functor $\LL\iota^*$ restricts to a functor $ \DA^{\eff}_{\Frobet}(k,\Lambda)\ra\sPerfDA^{\eff}_{\widehat{\A}^1}(k,\Lambda)$.
		\item The composition $\RR j_*\LL\iota^*$ is canonically equivalent to $\LL\Perf^*$ and restricts to a monoidal, triangulated equivalence of categories $ \DA^{\eff}_{\Frobet}(k,\Lambda)\cong\PerfDA^{\eff}(k,\Lambda)$.
	\end{enumerate}
The facts above prove the first claim. The commutativity of the diagram in the statement follows easily from the following commutative diagram, where we omit $\Lambda$ for brevity.
$$\xymatrix{
		 	\DA^{\eff}_{\et}(k)\ar[r]^{\LL \iota^*}	&\sPerfDA^{\eff}_{\A^1}(k)\ar[r]&\sPerfDA^{\eff}_{\widehat{\A}^1}(k)&	\PerfDA^{\eff}(k)\ar[l]\ar[l]_{\LL j^*}\\
	\FormDA^{\eff}_{\et}(K^\circ)\ar[r]^{\LL \iota^*}\ar[d]\ar[u]^{\sim}	&  \sPerfDA^{\eff}_{\mfB^1}(K^\circ)\ar[d]\ar[u]\ar[r]	&\sPerfDA^{\eff}_{\widehat{\mfB}^1}(K^\circ)\ar[d]\ar[u]^{\sim}	&	\PerfDA^{\eff}(K^\circ)\ar[l]\ar[l]_{\LL j^*}\ar[d]\ar[u]^{\sim}\\
	\RigDA^{\eff}_{\et}(K)\ar[r]^{\LL \iota^*}	& \sPerfDA^{\eff}_{\B^1}(K)		\ar[r]&\sPerfDA_{\widehat{\B}^1}^{\eff}(K)&	\PerfDA^{\eff}(K)\ar[l]_{\LL j^*}\\	
}
$$

From the diagram above, we also deduce that for the second claim, it is sufficient to provide a natural transformation  $\LL(\cdot)_\eta^*\RR j_*\Rightarrow\RR j_*\LL(\cdot)_\eta^*$. % 
This transformation can be obtained by adjunction from the transformation:
\[
\LL j^*\circ\LL(\cdot)_\eta^*\circ\RR j_*\cong \LL(\cdot)_\eta^*\circ \LL j^*\circ\RR j_*\Rightarrow\LL(\cdot)_\eta^*.
\]

We now prove the last point. In case $\car K=p$    it boils down to the commutation between $\Perf$ and the base change  thanks to assertion in  \eqref{1}.

We then assume $\car K=0$. In order to prove that a natural transformation is invertible, it suffices to test it on a specific set of generators of the category $\DA^{\eff}_{\Frobet}(k,\Lambda)$ such as motives $M=\Lambda(\mfX_\sigma)$ associated to special fibers of formal schemes $\mfX$ endowed with an \'etale map over $\mfT^N$. We now fix such an $\mfX$.

By means of Theorem \ref{dd} we have that $\RR(\cdot)_{\sigma*}M\cong\Lambda(\mfX)$. Similarly, since $(\mfX\times_{\mfT^N}\widehat{\mfT}^N)_\sigma\cong\mfX_\sigma^{\Perf}$  we deduce that $\RR(\cdot)_{\sigma*}\Lambda(\mfX_\sigma^{\Perf})\cong\Lambda(\mfX\times_{\mfT^N}\widehat{\mfT}^N))$. We then conclude the following sequence of isomorphisms:
$$\xi^*\LL\Perf^*M\cong\LL(\cdot)_\eta^*\RR(\cdot)_{\sigma*}\Lambda(\mfX_\sigma^{\Perf})\cong\LL(\cdot)_\eta^*\Lambda(\mfX\times_{\mfT^N}\widehat{\mfT}^N))\cong\Lambda(\mfX_\eta\times_{\T^N}\widehat{\T}^N).$$ This object is isomorphic, by means of the transformation above, to  
$$\RR j_*\LL\iota^*\xi M\cong\RR j_*\LL\iota^*\LL(\cdot)_{\eta}^*\R(\cdot)_{\sigma*}M\cong\mfG\Lambda(\mfX_\eta)$$
 as shown in \cite[Proposition 5.4 and Corollary 7.15]{vezz-fw}.
\end{proof}

\begin{prop}\label{premain2}
	The following diagram is commutative, up to a canonical transformation.
	\[
	\xymatrix{
		&\PerfDA^{\eff}(k,\Lambda)\ar[dl]_{\xi}\ar[dr]^{\xi^\flat}\\
		\PerfDA^{\eff}(K,\Lambda)%
		\ar@{<->}[rr]^{\sim}&&\PerfDA^{\eff}(K^\flat,\Lambda)%
	}
	\]
\end{prop}

\begin{proof}
	We consider the following diagram
	\[
	\xymatrix{
		&\PerfDA^{\eff}(k,\Lambda)\\
		\PerfDA^{\eff}(K^\circ,\Lambda)\ar[d]_{\LL(\cdot)_\eta^{*}}\ar[ur]_{\sim}^{\LL(\cdot)_\sigma^{*}}\ar[rr]^{\LL(\cdot)^\flat}&&
		\PerfDA^{\eff}(K^{\flat\circ},\Lambda)\ar[ul]^{\sim}_{\LL(\cdot)_\sigma^{*}}\ar[d]_{\LL(\cdot)_\eta^{*}}
		\\
		\PerfDA^{\eff}(K,\Lambda)\ar@{<->}[rr]^{\sim}&&\PerfDA^{\eff}(K^\flat,\Lambda)
	}
	\]
	where the functor $\LL(\cdot)^\flat$ is induced by $A\mapsto A^\flat=\varprojlim A/p$. It sends smooth perfectoid formal  spaces over $K^\circ $ to the same objects over $K^{\flat\circ}$ and preserves the \'etale topology. 
	
	By definition, this functor makes the triangle as well as the rectangle of the diagram commute. On the other hand, since the two sides of the triangle are equivalences, we deduce that it also is.
	
	We conclude that there is an invertible natural transformation from 
	$\xi$ to $\xi^\flat$ 
	as wanted. %
\end{proof}

\begin{proof}[Proof of Theorem \ref{main}]
We combine the diagrams of Theorem \ref{premain} for $K$ and $K^\flat$ with the diagram of Proposition \ref{premain2} obtaining the following (we can assume $\car K=0$ and we omit $\Lambda$):
$$
\xymatrix{
	\DA^{\eff}_{\Frobet}(k)\ar[d] \ar[r]^{\LL\Perf^*}_{\sim}	 \ar@/^35pt/[rrr]^{\id}&	\PerfDA^{\eff}(k)\ar[d]\ar@{=}[r]&\PerfDA^{\eff}(k)\ar[d]&\DA^{\eff}_{\Frobet}(k)\ar[d]	\ar[l]_{\LL\Perf^*}^{\sim}\\
	\RigDA^{\eff}(K)	\ar[r]^{\sim} \ar@/_35pt/[rrr]^{\mfG_{K,K^\flat}}&	\PerfDA^{\eff}(K)&\PerfDA^{\eff}(K^\flat)\ar@{<->}[l]_{\sim}&\RigDA^{\eff}_{\Frobet}(K^\flat)\ar[l]_{\sim}	\\	
}
$$
\end{proof}

As stated in the introduction, this commutativity gives an explicit formula for tilting motives of varieties of good reduction.

\begin{cor}Let $M$ be in $\FormDA^{\eff}(K^\circ,\Lambda)$. 
	The tilt of the motive  $\LL(\cdot)_\eta^*M$ is the motive $\LL(\cdot\widehat{\times}_{\Spec k}\Spa K^\flat)\LL(\cdot)_{\sigma}^*M$. In particular, the tilt of the motive $\Lambda(\mfX_\eta)$ of a rigid analytic variety of good reduction  is the motive associated to $\mfX_\sigma\times_k \Spa K^\flat$ that is, the analytification  of the special fiber  base-changed to $K^\flat$. It is also isomorphic to the motive of $Q_{\rig}(\mfX_\sigma)$. % 
\end{cor}

\begin{proof}
By Theorem \ref{dd} the motive $\LL(\cdot)_\eta^*M$ is canonically isomorphic to $\xi \LL(\cdot)_\sigma^*M$ and hence its tilt is isomorphic to $\xi^\flat \LL(\cdot)_\sigma^*M$ by Theorem \ref{main}. The statement follows from the explicit description of $\xi^\flat$ given in Proposition \ref{chiflat}.
\end{proof}

We conclude this section by mentioning the stable version of Theorem \ref{main}. By formally inverting the Tate shift $\Lambda(1)$ it is possible to define the categories $\DM(k,\Lambda)$ and $\RigDM(K,\Lambda)$. We refer to \cite[Definition 2.5.27]{ayoub-rig} for their formal definition. Under the assumption $\Q\subset\Lambda$ these categories  fully faithfully contain the effective categories of motives $\DM^{\eff}(k,\Lambda)$ and $\RigDM^{\eff}(K,\Lambda)$ (see \cite[2.5.49]{ayoub-rig}) and are generated (as triangulated categories with small sums) by the Tate twists of effective motives.

\begin{rmk}\label{DADMstable}
Also in the stable setting, adding transfers or not makes no difference  as long as we consider the $\Frobet$-topology  and the ring of coefficients $\Lambda$ contains $\Q$  \cite{vezz-DADM}. In the stable algebraic setting, the equivalence $\DM(k,\Lambda)\cong \DA(k,\Lambda)$ holds already at the level of \'etale motives, as long as the characteristic of $k$ is invertible in $\Lambda$ (see \cite[Appendix B]{ayoub-etale}). We can therefore write alternatively $\DA(k,\Lambda)$ or $\DM(k,\Lambda)$, $\RigDA_{\Frobet}(K,\Lambda)$ or $\RigDM(K,\Lambda)$.
\end{rmk}

\begin{cor}
The following diagram is commutative, up to  a natural transformation, and extends the triangle of Theorem \ref{main}.
\[
\xymatrix{
	&\DM(k,\Lambda)\ar[dl]_{\xi}\ar[dr]^{\xi^\flat}\\
	\RigDM(K,\Lambda) %
	\ar@{<->}[rr]^{\sim}_{\mfG'}&&\RigDM(K^\flat,\Lambda) %
}
\]
\end{cor}

\begin{proof}
The diagrams and the natural transformations of Theorems \ref{premain} and \ref{premain2} can be constructed also for stable motives, since the left Quillen functors which appear there preserve the Tate shift. We can then adopt the same strategy as in the proof of Theorem \ref{main}: in order to prove that a natural transformation of functors between stable motives is invertible, it is sufficient to test it on Tate twists of effective ones. As the functors involved preserve the Tate twist, the result follows from the check made in Theorem \ref{main}.
\end{proof}

\section{de Rham cohomology over local fields of positive characteristic}
\label{secdR}

Suppose that $K$ is a complete non-archimedean field of characteristic $0$. In order to define a cohomology theory ``\`a la de Rham'' for rigid analytic varieties over $K$ one has to endow them with the structure $X^\dagger$ of a so-called \emph{dagger} \emph{space}: this amounts (at least locally) to a choice of an embedding of $X$ inside the interior of a larger variety $X'$. One then needs to consider only those differential forms which extend in some strict neighbourhood of $X$ in $X'$ and consider the associated cohomology groups. We refer to \cite{gk-dR} for the definitions of dagger spaces and the resulting cohomology theory, called ``overconvergent de Rham cohomology''. 

In \cite{vezz-rigidreal} we proved that the overconvergent de Rham cohomology of rigid analytic varieties over $K$ is represented by a motive of $\RigDA^{\eff}(K,\Lambda)$. We briefly recall here its construction. We can consider the complex of overconvergent differential forms $\Omega^\dagger$ which is a complex of presheaves on smooth dagger spaces. It induces an object of the category of dagger motives $\RigDA^{\dagger\eff}(K,\Lambda)$ when $K\subset\Lambda$ and hence a motive of $\RigDM^{\eff}(K,\Lambda)$ by means of the canonical equivalence
\[
\LL l^*\colon \RigDA^{\dagger\eff}(K,\Lambda)\cong\RigDA^{\eff}(K,\Lambda)\cong\RigDM^{\eff}(K,\Lambda).
\]
By abuse of notation, we will still denote the motive $\LL l^*\Omega^\dagger$ by $\Omega^\dagger$. 

\begin{rmk}\label{Frobac}
We show in \cite{vezz-rigidreal} that the induced cohomology theory defined on $\DM(k,\Lambda)$ via the functor $\xi\colon\DM(k,\Lambda)\ra\RigDM(K,\Lambda)$ and the complex $\Omega^\dagger$ is the usual rigid cohomology. We remark that by functoriality of this construction and the fact that the relative Frobenius is invertible in $\DM(k,\Lambda)$ (see \cite{vezz-DADM} for example) this cohomology theory gains automatically a canonical action of Frobenius in case $k$ is a finite field. The same fact holds for formal motives, and formal perfectoid motives, beacuse of the equivalences $\FormDA(K^\circ,\Lambda)\cong\DA(k,\Lambda)\cong\PerfDA(K^\circ,\Lambda)$. 
\end{rmk}

We now suppose that $K$ is perfectoid. If we let $K^\flat$ be its tilt, we will say that $K$ is an \emph{untilt} of $K^\flat$. Tilting is a canonical operation, but untilting is not: the various untilts of $K^\flat$ are  parametrized by  principal ideals of $W(\mcO_{K^\flat})$ generated by a primitive element of degree one (see \cite{fontaine-bourbaki}.

By using the rigid analytic tilting equivalence, the motive $\Omega^\dagger$ represents a cohomology theory also on rigid analytic varieties over $K^\flat$ as well as on smooth perfectoid spaces over $K$ (or $K^\flat$). We now list the formal properties of such cohomology theories, and their relationship with rigid cohomology of the special fibers, in  light of the results of Section \ref{secmain}. % 

\begin{dfn}
	Let $K'$ be a perfectoid field of any characteristic, and let $K$ be an untilt of $K'^\flat$ of mixed characteristic. Put $\Lambda=K$. Let $$\mfG_{K',K}\colon\RigDM^{\eff}(K',\Lambda)\stackrel{\sim}{\ra}\RigDM^{\eff}(K'^\flat,\Lambda)\stackrel{\sim}{\leftarrow}\RigDM^{\eff}(K,\Lambda)$$ be the canonical motivic tilting equivalence from rigid motives over $K'$ to the ones over $K$. 
	 The  \emph{$K$-de Rham (overconvergent) cohomology} $H^i_{\dR}(M,K)$ of a rigid analytic motive $M$ over $K'$ is the overconvergent de Rham cohomology of the motive $\mfG_{K',K}M$. 
	 
	 Let now $M$ be an algebraic motive in $\DM^{\eff}(K',\Lambda)$. The  \emph{$K$-de Rham (overconvergent) cohomology} $H^i_{\dR}(M,K)$ of  $M$  is the $K$-de Rham cohomology of the motive $\LL\Rig^*M$. 
\end{dfn}

\begin{rmk}
	The above definition specializes obviously to the following objects:
	\begin{enumerate}
	\item Smooth rigid analytic varieties $X$ over (any valued subfield of) $K'^\flat$: it suffices to consider $M=\Lambda(X)$.
	\item Arbitrary quasi-projective algebraic varieties $X$ over (any subfield of) $K'^\flat$: it suffices to consider $M=\LL\Rig^*\Lambda(X)$. If $X$ is smooth, this motive is $\Lambda(X^{\an})$.
	\end{enumerate}
\end{rmk}

\begin{rmk}
If $K$ has mixed characteristic, the $K$-de Rham cohomology over $\RigDM^{\eff}(K,\Lambda)$ coincides with the usual overconvergent de Rham cohomology, which in turn extends the usual de Rham cohomology of algebraic varieties (see \cite[Theorem 2.3]{gk-dR}). In particular,  the $K$-de Rham cohomology over  $\DM^{\eff}(K,\Lambda)$ is the usual algebraic de Rham cohomology with coefficients in $K$.
\end{rmk}

\begin{rmk}
	In general, the $K$-de Rham overconvergent cohomology of algebraic varieties over $K^\flat$ does not coincide with rigid cohomology:  indeed, its coefficients lie in $K$ which is way ``smaller" than a non-archimedean field with residue $K^\flat$. 
	We also point out that the topology of $K^\flat$ plays a crucial role in this definition, as it is not considered merely as an abstract field (the tilting equivalence is used crucially in the construction).
\end{rmk}

\begin{dfn}\label{cpt}
	An object $X$ of a triangulated category with small sums $\catT$ is \emph{compact} if for any small collection $\{Y_i\}$ of objects in $\catT$ one has  \[\Hom(X,\bigoplus Y_i)\cong\bigoplus\Hom(X,Y_i).\]
\end{dfn}

\begin{exm}
	The motive $\Lambda(X)$ of a quasi-compact smooth variety $X$ over $K$ is compact in $\RigDA^{\eff}(K,\Lambda)$ (see \cite[Proposition 1.2.34]{ayoub-rig}).
\end{exm}

\begin{prop}\label{dR}
	The $K$-de Rham cohomology $H^i_{\dR}(M,K)$ satisfies the following properties:
	\begin{enumerate}
		\item It is represented by the motive $\mfG_{K,K'}\Omega^\dagger$ i.e. $H^i_{\dR}(M,K)\cong H_i\Hom_\bullet(M,\mfG_{K.K'}\Omega^\dagger)$.
	\item If $M$ is compact, then it is finite dimensional  and equal to $0$ for $|i|\gg0$.
	\item It satisfies \'etale descent.
	\item It is homotopy-invariant.
	\item It is compatible with field extensions $L'/K'$.
	\item It satisfies the K\"unneth formula on compact motives i.e. $$H^n_{\dR}(M\otimes N,K)\cong\bigoplus_{i+j=n}H^i_{\dR}(M,K)\otimes H^j_{\dR}(N,K)$$
	for $M,N$ compact.
	\end{enumerate}
\end{prop}

\begin{proof}
By construction, it suffices to prove the statement for the $K$-de Rham cohomology of analytic $K$-varieties, that is the overconvergent de Rham cohomology with coefficients in $K$. All the properties above are proved in \cite[Section 5]{vezz-rigidreal}.
\end{proof}

\begin{rmk}
	By means of the equivalence $\PerfDA(K,\Lambda)\cong\RigDA(K,\Lambda)$ we can define a $K'$-de Rham cohomology for smooth perfectoid spaces (that is, perfectoid spaces which are locally \'etale over a perfectoid poly-disc) enjoying all the properties listed in Proposition \ref{dR}.
\end{rmk}

The following corollary specifies how this cohomology theory is an extension of rigid cohomology.

\begin{cor}
	The $K$-de Rham cohomology of a rigid analytic variety $X$ over $K^\flat$ with good reduction coincides with the rigid cohomology of its reduction $\bar{X}$ with coefficients in $K$.
	\end{cor}
	
	\begin{proof}
Let $\mfX$ be a smooth formal model of $X$. From Theorems \ref{dd} and \ref{main} we obtain the following isomorphisms: $$\mfG_{K^\flat,K}\Lambda(X)\cong\mfG_{K^\flat,K}\LL(\cdot)_\eta^*\Lambda(\mfX)\cong\mfG_{K^\flat,K}\xi\LL(\cdot)_\sigma^*\Lambda(\mfX)\cong\xi\Lambda(\mfX_\sigma).$$
We already showed in \cite{vezz-rigidreal} that the overconvergent de Rham cohomology of a motive of the form $\xi M$ coincides with the rigid cohomology of $M$ hence the claim.
	\end{proof}
	
	\begin{rmk}\label{riganddeR}
More generally, we proved that the $K$-de Rham cohomology of a rigid analytic  motive over $K^\flat$ of the form $\LL(\cdot)_\eta^*M$  coincides with the rigid cohomology of its reduction $\LL(\cdot)_\sigma^*M$ with coefficients in $K$. 
	\end{rmk}

		The following corollary provides an alternative method to define and compute rigid cohomology. In order to pass from the characteristic $p$ side to the characteristic $0$ side, we can either find a formal model over a valuation ring of mixed characteristic and then consider its rigid analytic generic fiber (this is the classic recipe of Monsky and Washnitzer) we can alternatively take the perfection, and then use Scholze's tilting, at the cost of enlarging the ring of coefficients. We recall that the chosen section $k\ra K^{\flat\circ}$ induces a map $k\ra K^{\circ}/p$ and hence a map $W(k)\ra K^{\circ}$.
		
		\begin{cor}
		\label{rigtilt}
		Let $k$ be the residue field of a perfectoid field $K$. 
Up to a base change along $W(k)[1/p]\ra K$, we can compute the rigid cohomology of a smooth variety $\bar{X}$ over $k$ as  the $K$-de Rham cohomology of $\bar{X}_{K^\flat}$.
		\end{cor}
		
		\begin{proof}
It suffices to read the equivalence of Remark \ref{riganddeR} backwards.
		\end{proof}

\begin{rmk}
In the above definition of a de Rham cohomology for    $X/K^\flat$ a choice of an untilt of $K^\flat$ is necessary. It is expected to ``globalize'' this definition and to associate to a variety $X/K^\flat$ a vector bundle over the Fargues-Fontaine curve related to $K^\flat$ such that its stalk at a point $x_K$ corresponding to an untilt $K$ of $K^\flat$ coincides with $H^i_{dR}(X,K)$.
\end{rmk}

\section{On a conjecture of Ayoub}
\label{secconj}

Along this section, we make the following hypotheses. In particular, we no longer require $K$ to be perfectoid.

\begin{assu}
	We let $K$ be a complete valued field with respect to a non-archimedean valuation of rank $1$. We let $K^\circ$ be its ring of integers and $k$ its residue field, which we assume to have characteristic $p>0$. We also suppose that $\Lambda$ is a $\Q$-algebra. 
\end{assu}

In the final sections of \cite{ayoub-rig} Ayoub proves a series of results on the relations between the categories of rigid motives $\RigDM(K,\Lambda)$ and some categories of motives defined over the residue field $k$. The most striking result holds for $K\cong\C(\!(t)\!)$ where $\RigDM(K,\Lambda)$ is shown to embed in $\DM(\G_{m,\C},\Lambda)$ which is the category of  motives over the base $\G_{m,\C}$ (see \cite[Section 11]{cd}). This fact allowed Ayoub to define  a motivic version of the decomposition Galois group as well as some important non-obvious applications  (\cite{ayoub-h2}, \cite{ayoub-conj} and \cite{ais}).

The case of positive residue characteristic is more delicate. An analogous, yet weaker, statement for fields of the form $k(\!(t)\!)$ with $\car k=p$ is still proved in \cite{ayoub-rig} and deals with rigid analytic motives \emph{of good reduction}. We now recall it, and prove its analogue in the mixed characteristic case, which was conjectured by Ayoub. It is surely hoped to adapt the motivic Galois group-theoretical machinery to this situation in order to produce some generalizations of the consequences obtained in the equal characteristic $0$ case.

\begin{rmk}
In this section, we will consider complete valued fields which may not be perfect. Therefore, we can not refer to Theorem \ref{DADM}. All our statements will be given in terms of motives with transfers,  even though in proofs, motives without transfers will also be used.
\end{rmk}

\begin{dfn}
We denote by $\RigDM^{\eff,\gr}(K,\Lambda)$ [resp.  $\RigDM^{\gr}(K,\Lambda)$] the  triangulated subcategory of $\RigDM^{\eff}(K,\Lambda)$ [resp.  $\RigDM(K,\Lambda)$] closed by small sums generated by [Tate twists of] motives of the form $\Lambda_{\tr}(\mfX_\eta)$ with $\mfX$ a smooth affine formal scheme, topologically of finite type over $K^\circ$. Their objects are called \emph{motives of good reduction}.
\end{dfn}

\begin{rmk}
Since the motives of the form $\Lambda(\mfX)$ with $\mfX$ affine and smooth generate $\FormDA^{\eff}(K^\circ,\Lambda)$ as a triangulated category with small sums, we deduce that the category $\RigDM^{\eff,\gr}(K,\Lambda)$ [resp.  $\RigDM^{\gr}(K,\Lambda)$]  is equally generated by the motives of the form $\LL a_{\tr}\circ\LL(\cdot)_\eta M$ with $M$ in $\FormDA^{\eff}(K^\circ,\Lambda)$ [resp. in $\FormDA(K^\circ,\Lambda)$] where $\LL a_{\tr}\colon\RigDA^{\eff}(K,\Lambda)\ra\RigDM^{\eff}(K,\Lambda)$ is the ``adding transfers" functor, that is the (left Quillen derived) functor  induced by the inclusion $\RigSm/K\ra\RigCor/K$.
\end{rmk}

\begin{rmk}
The category $\RigDM^{\gr}(K,\Lambda)$ is denoted by $\RigDM^{\br}(K,\Lambda)$ in \cite{ayoub-rig} according to the French terminology \emph{bonne reduction}.
\end{rmk}

\begin{rmk}\label{allpotgr}
Having a description of the category $\RigDM^{\eff,\gr}(K,\Lambda)$ has important consequences also for arbitrary compact motives in $\RigDM^{\eff}(K,\Lambda)$. Indeed, as shown in \cite[Theorem 2.5.34]{ayoub-rig} \emph{any} compact motive $M$ is of good reduction, up to a finite  extension $K'/K$ of the base field. %
\end{rmk}

Along this section, we will make constant use of the following theorem, proven by Ayoub. It will allow us to restrict to compact objects in many proofs, which is technically convenient (see for example \cite[Lemma 1.3.32]{ayoub-rig}).

\begin{thm}[{\cite[Proposition 1.2.34]{ayoub-rig}}]
	The motive $\Lambda(X)$ of a smooth quasi-compact rigid analytic variety $X$ is compact in $\RigDA^{\eff}_{\et}(K,\Lambda)$. Such motives form a set of compact generators for $\RigDM^{\eff,\gr}(K,\Lambda)$.%  
\end{thm}

The theorem implies that similar statements also hold also in the category $\RigDM^{\eff}(K,\Lambda)$ as well as in their stable counterpart. One may also consider smooth \emph{affinoid} varieties as a set of generators (as any smooth variety is locally affinoid). 

We now consider compact motives in the category $\RigDM^{\eff,\gr}(K,\Lambda)$.

\begin{prop}\label{cpt2}
An object $M$ of $\RigDM^{\eff,\gr}(K,\Lambda)$ [resp.  $\RigDM^{\gr}(K,\Lambda)$] is compact if and only if it is compact in the bigger category $\RigDM^{\eff}(K,\Lambda)$ [resp.  $\RigDM(K,\Lambda)$].
\end{prop}

\begin{proof}
By \cite[Theorem 2.1.24]{ayoub-th1} we obtain that the triangulated subcategory of compact objects in $\RigDM^{\eff,\gr}(K,\Lambda)$ [resp.  $\RigDM^{\gr}(K,\Lambda)$] is generated by the [twists of]  direct factors of motives of the form $\Lambda_{\tr}(\mfX_\eta)$ with $\mfX$ an affine smooth formal scheme. Such motives are also compact in $\RigDM^{\eff}(K,\Lambda)$ hence the claim.
\end{proof}

\begin{rmk}
Thanks to the previous proposition, the class of compact motives of $\RigDM(K,\Lambda)$ which are also of good reduction coincide with the class of compact objects in the subcategory $\RigDM^{\gr}(K,\Lambda)$ of motives of good reduction. We can then unambiguously refer to them as ``compact motives of good reduction''.
\end{rmk}

The following theorem by Ayoub  amounts to a description  of rigid motives of good reduction in terms of algebraic motives over the residue field. We recall that $\Lambda$ is supposed to be a $\Q$-algebra.

\begin{dfn}
Let $S$ be an algebraic variety over a field $k$ and let $\Lambda$ be a ring. We denote by $\DM^{\eff}(S,\Lambda)$ [resp. $\DM(S,\Lambda)$] the category of effective motives [resp. motives] over $S$ following Voevodsky. We denote by $\Un\DM^{\eff}(k,\Lambda)$ [resp. $\Un\DM(k,\Lambda)$] the triangulated subcategory with small sums of $\DM^{\eff}(\G_{m,k}^n,\Lambda)$ [resp. $\DM(\G_{m,k}^n,\Lambda)$] generated by [twists of] motives of the varieties of the form $\G_{m,X}^n$ with $X$ smooth over $k$.
\end{dfn}

\begin{thm}[{\cite[Theorem 2.5.57]{ayoub-rig}}]\label{ayoubgr}
	Let $K$ be a complete valued field of equal characteristic, with residue field $k$ and a value group which is free of rank $n$ over $\Z$. Fix a section $k\ra K^\circ$ of the reduction and some elements $\pi_1,\ldots,\pi_n$ of $K^\circ$ whose values generate $|K^\times|$. If the characteristic of $K$ is positive, we also suppose that $k[\pi_1^{\pm1},\ldots,\pi_n^{\pm1}]$ is dense in $K$. 
The composite functors
\[
\xymatrix @R=.2pc{
\Un\DM^{\eff}(k,\Lambda)\ar@{^{(}->}[r] &\DM^{\eff}(\G^n_{m,k},\Lambda)\ar[rr]^{(\pi_1,\ldots,\pi_n)^*}&&\DM^{\eff}(K,\Lambda)\ar[r]^-{\Rig^*}&\RigDM^{\eff}(K,\Lambda)\\
\Un\DM(k,\Lambda)\ar@{^{(}->}[r] &\DM(\G^n_{m,k},\Lambda)\ar[rr]^{(\pi_1,\ldots,\pi_n)^*}&&\DM(K,\Lambda)\ar[r]^-{\Rig^*}&\RigDM(K,\Lambda)
}\]
take values in the subcategories of motives of good reduction, and induce monoidal, triangulated equivalences
\[
\Un\DM^{\eff}(k,\Lambda)\cong \RigDM^{\eff,\gr}(K,\Lambda)\quad  \Un\DM(k,\Lambda)\cong \RigDM^{\gr}(K,\Lambda)
\]
\end{thm}

\begin{rmk}
We point out that the fact above has important applications also for arbitrary compact motives  (not necessarily of good reduction). Indeed, following Remark \ref{allpotgr} we deduce that for any pair of compact motives $M,M'$ we can calculate $\Hom_{\RigDM(K,\Lambda)}(M,M')$ in terms of the Galois-invariants in a Hom-group of $\DM(\G^n_{m,k'},\Lambda)$ for some finite Galois extension $K'/K$ (here $k'$ is the residue field of $K'$) see \cite[Remark 2.5.71]{ayoub-rig}.
\end{rmk}

\begin{rmk}\label{chi1}
We recall that we denote by $\chi$  the functor  $$\chi\colon \RigDM(K,\Lambda)\ra\DA(k,\Lambda)\cong\FormDA(K^\circ,\Lambda)$$ which is right adjoint to the functor $\LL(\cdot)_\eta\colon \FormDA(K^\circ,\Lambda)\ra\RigDM(K,\Lambda)$ induced by the generic fiber functor. The previous result allows to compute $\chi$ on motives of good reduction in an alternative way, as the following composition:
$$
\RigDM^{\gr}(K,\Lambda)\cong \Un\DM(k,\Lambda)\subset\DM(\G_{m,k},\Lambda)\stackrel{P_*}{\ra}\DM(k,\Lambda)
$$with $P\colon\G_{m,k}\ra\Spec k$ being the structure morphism. In particular $\chi\Lambda\cong \Lambda\oplus\Lambda(-1)[-2]$.
\end{rmk}

Ayoub conjectured that the  last two equivalences of the theorem above still hold true in the mixed characteristic case, under some technical hypotheses on $K$ (see Assumption \ref{assK}). 

Our strategy is simple, and we first sketch it  for the case of $\Q_p$ (the general case will be proved in Theorem \ref{ayoubgr+}). In this case, the conjecture takes the following form.
\begin{thm}
There are triangulated, monoidal equivalences of categories
\[
\UU\DM^{\eff}(\F_p,\Lambda)\cong \RigDM^{\eff,\gr}(\Q_p,\Lambda)\quad  \UU\DM(\F_p,\Lambda)\cong \RigDM^{\gr}(\Q_p,\Lambda)
\]
\end{thm}

Let $\widehat{K}$ be the completion of the field $\Q_p(\mu_{p^\infty})$. It is a perfectoid field with tilt $\widehat{K}^\flat$ isomorphic to the completion of $\F_p(\!(t)\!)(t^{1/p^\infty})$. This field is the perfection of $K^\flat\colonequals \F_p(\!(t)\!)$ which is a field that satisfies the assumptions of Theorem \ref{ayoubgr} and has the same residue field $\F_p$ as $\Q_p$. The analogue of the theorem above for $K=\Q_p$ will then be reached via the following diagram of equivalences:%
\begin{equation}\label{diagQp}
\xymatrix{
	\RigDM^{\gr}(K,\Lambda)\ar[d]_{\sim} & \RigDM^{\gr}(K^\flat,\Lambda)\ar[d]^{\sim}\\
		\RigDM^{\gr}(\widehat{K},\Lambda)\ar@{<->}[r]^{\sim} & \RigDM^{\gr}(\widehat{K}^\flat,\Lambda)\\	
	}
	\end{equation}

We now introduce Ayoub's conjecture for a general field $K$ of mixed characteristic (not just $\Q_p$). We recall that in the case of equal characteristic $p$  we assumed that $k[\pi_1^{\pm1},\ldots,\pi_n^{\pm1}]$ is dense in $K$. Similarly, in the case of a mixed characteristic field we  need to consider some special assumptions on the valuation, that we list below:

\begin{assu}\label{assK}
	From now on, we fix a complete, non-archimedean field $K$ over $\Q_p$  with the following extra properties:
	\begin{enumerate}[(i)]
		\item There are $n$ elements $\pi_1,\ldots,\pi_n$ of $K^\circ$ with $\pi_1$ algebraic over $\Q_p$ such that the group $|K^\times|$ is free of rank $n$ generated by the valuations of $\pi_i$.
		\item There is a complete discrete valued subfield $K_0$ having $\pi_1$ as uniformizer and the same residue field $k$ as $K$.
		\item The ring $K_0[\pi_1^{\pm1},\ldots,\pi_n^{\pm1}]$ is dense in $K$.
	\end{enumerate}
\end{assu}

\begin{rmk}
	Any finite field extension of $\Q_p$ obviously satisfies Assumption \ref{assK}.% 
\end{rmk}

In the case of $\Q_p$ we introduced the intermediate, auxiliary fields $\Q_p(\mu_{p^\infty})$ and $\F_p(\!(t)\!)$. We now try to generalize this situation. 
\begin{dfn}If $\car K=0$ we say $\widehat{K}$ is a [Galois] \emph{completed perfection} of $K$ if it is a perfectoid field obtained as the completion of a [Galois] algebraic extension  of $K$ inducing purely inseparable extensions of residue fields. %
\end{dfn}

\begin{prop}\label{sugiu}
	Let $K$ satisfy Assumption \ref{assK}. There exists a Galois completed perfection $\widehat{K}$ of $K$ and a complete valued field  $K^\flat$ of positive characteristic with the following properties:
	\begin{enumerate}[(i)]
		\item  $\widehat{K}$ has a value group which is a free module over $\Z[1/p]$ of rank $n$.
		\item  $\widehat{K}^\flat$ is the completed perfection of  $K^\flat$.
		\item $K^\flat$ has a free value group of rank $n$ over $\Z$ and residue field $k$.%
		\item  We can choose  some elements $\pi_1^\flat,\ldots,\pi_n^\flat$ in $K^\flat$ such that their values $|\pi_i^\flat|$ generate the value group, and such that  %
		the ring $k[(\pi_i^{\flat})^{\pm1}]$ is dense in $K^\flat$.
	\end{enumerate}
\end{prop}

\begin{proof}	
		Choose some representatives of a finite subset of  $k^\times=(K^\circ_0/\pi_1)^\times$ in $K_0$ and some $p$-th roots of them in a complete algebraic closure $C$ of $K$. The extensions of $K$ and $K_0$ generated by these elements satisfy Assumption \ref{assK} with respect to the same elements $\pi_i$ and the same value group, and they induce a finite purely inseparable map on the residue fields. %
	With a direct limit process, we obtain a pair $(F,F_0)$  which is algebraic over $(K,K_0)$ and whose completion satisfies Assumption \ref{assK} with the same value group of $K$ and  residue field isomorphic to $k^{\Perf}$.  %
	
	We now take $L$ to be the completion of $F(\mu_{p^\infty})$. It is the completion of a Galois extension over $K$ and its subfield $L_0\colonequals F_0(\mu_{p^\infty})^\wedge$ is perfectoid, with tilt isomorphic to $k^{\perf}(\!(t)\!)(t^{1/p^\infty})^\wedge$ and value group isomorphic to $\Z[1/p]$ generated by the valuation of $\widetilde{\pi}_1\colonequals t^\sharp$ (for the definition of the multiplicative map $x\mapsto x^\sharp$, see \cite[Lemma 3.4]{scholze}). %
	
	 We finally consider $\widehat{K}$ to be the completion of $L(\pi_2^{1/p^\infty},\ldots,\pi_n^{1/p^\infty})$ for a choice of compatible $p$-th roots of each $\pi_i$ with $i>1$. Its value group is free over $\Z[1/p]$ generated by the values of $\widetilde{\pi}_1$ and the $\pi_i$'s with $i>1$. In order to prove that it is perfectoid, it suffices to show that the Frobenius is surjective on the ring $\widehat{K}^{\circ}/\pi_1$  which by density coincides with $L_0[\pi_i^{\pm1/p^\infty}]^\circ/\widetilde{\pi}_1$. This follows from the fact that  $L_0^\circ/\widetilde{\pi}_1$ is perfect, and  any element of the ring $L_0[\pi_i^{\pm1/p^\infty}]^\circ$  is a sum of terms of the form $u\widetilde{\pi}_1^{q_1}\pi_2^{q_2}\cdot\ldots\cdot \pi_n^{q_n}$ with $u\in L_0^{\circ\times}$ and $q_i\in\Z[1/p]$. Also the field $\widehat{K}$ is the completion of a Galois extension over $K$.% 
	
	We now put $K^\flat$ to be the completion of the subfield  $k(\pi_1^\flat,\ldots,\pi_n^\flat)$ of $\widehat{K}^{\flat}$ where $\pi_1^\flat=t$ and the other elements $\pi_i^\flat$ are chosen to satisfy $|\pi_i^{\flat\sharp}|=|\pi_i|$. In particular, the values $|\pi_i^\flat|$ generate the value group of $K^\flat$ which is a free $\Z$-module of rank $n$ and the residue field of $K^\flat$ is $k$. %
	
	In order to see that the completed perfection of $K^\flat$ coincides with $\widehat{K}^\flat$ it suffices to show that $k^{\Perf}[(\pi_i^{\flat})^{\pm1/p^\infty}]^\circ$ is dense in $\widehat{K}^{\flat\circ}$ i.e. that for any $n$ and any $\lambda\in \widehat{K}^{\flat\circ}$ there exists an element $\xi\in k^{\Perf}[(\pi_i^{\flat})^{\pm1/p^\infty}]^\circ$ such that $\lambda-\xi\in\pi_1^{\flat n}\widehat{K}^{\flat\circ}$. By induction on $n$ it suffices to prove this statement for $n=1$. %
	We obtain the following isomorphisms:
	\[
	k^{\Perf}[(\pi_i^{\flat})^{\pm1/p^\infty}]_{i\geq1}^{\circ}/\pi_1^\flat\cong L_0[\pi_i^{\pm1/p^\infty}]_{i\geq2}^{\circ}/\widetilde{\pi}_1\cong \widehat{K}^\circ/\widetilde{\pi}_1\cong \widehat{K}^{\flat\circ}/\pi_1^\flat
	\]
	proving our claim.
\end{proof}

We start by proving that the tilting equivalence respects motives of good reduction: that is, we prove the bottom horizontal equivalence of the diagram \eqref{diagQp}. In order to do this, we use the results of the previous sections.%

\begin{prop}\label{tiltgr}
Let $\widehat{K}$ be a perfectoid field. % 
	The motivic tilting equivalence of Theorem \ref{mottilteq} restricts to  equivalences \[\RigDM^{\eff,\gr}(\widehat{K},\Lambda)\cong\RigDM^{\eff,\gr}(\widehat{K}^\flat,\Lambda)\quad \RigDM^{\gr}(\widehat{K},\Lambda)\cong\RigDM^{\gr}(\widehat{K}^\flat,\Lambda)\]
\end{prop}

\begin{proof}
	It suffices to show that motives of the form $\LL(\cdot)_\eta M$ are sent by $\mfG$ to motives of the same form, in the two directions. 
	
	This follows from the following commutative diagram (up to a natural transformation) which appears in Theorem \ref{main}.
	\[
	\xymatrix{
	\DA^{\eff}(k,\Lambda)\ar@{=}[d]&\FormDA^{\eff}(K^\circ,\Lambda)\ar[l]_-{\sim}\ar[r]^-{\LL (\cdot)_\eta}&\RigDA^{\eff}(K,\Lambda)\ar@<-6ex>[d]^{\sim}\cong\RigDM^{\eff}(K,\Lambda)\\
		\DA^{\eff}(k,\Lambda)&\FormDA^{\eff}(K^{\flat\circ},\Lambda)\ar[l]_-{\sim}\ar[r]^-{\LL (\cdot)_\eta}&\RigDA^{\eff}_{\Frobet}(K^\flat,\Lambda)\cong\RigDM^{\eff}(K^\flat,\Lambda)\\
	}
	\]
\end{proof}

We now show how to obtain the right vertical equivalence of the diagram \eqref{diagQp}.

\begin{prop}\label{K=Kperf}
	Suppose $\car K=p$ and let $\widehat{K}$ be its completed perfection. The base-change functor $(\widehat{K}/K)^*$ induces triangulated, monoidal equivalences:  \[\RigDM^{\eff}(K,\Lambda)\cong\RigDM^{\eff}(\hat{K},\Lambda)\qquad \RigDM(K,\Lambda)\cong\RigDM(\widehat{K},\Lambda).\] These equivalences restrict to the subcategories of motives of good reduction.
\end{prop}

\begin{proof}
	Write $\widehat{K}$ as a completion of the direct limit of some finite, purely inseparable extensions of $K$. It suffices to show the equivalence for compact motives, which follows from Lemmas \ref{2lim} and \ref{2lim2} and the remark that any finite, purely inseparable extension $K'/K$ induces an equivalence $\RigDM^{\eff}(K',\Lambda)\cong\RigDM^{\eff}({K},\Lambda)$ (see \cite[Proposition 2.2.22]{ayoub-rig}). The fact that these equivalences  preserve motives of good reduction follows from the proof of \cite[Propositions 2.2.22 and 2.2.37]{ayoub-rig}: it suffices to prove that the pullback along Frobenius $\Frob\colon K'\ra K'$ induces an essentially surjective functor on compact motives of good reduction, and this follows from the fact that for any smooth formal scheme $\mfX$ over $K'^\circ$, the relative Frobenius map $\mfX_\eta\ra(\mfX{\times}_{\Frob}K'^\circ)_\eta$ is invertible as a correspondence with $\Q$-coefficients.
\end{proof}

\begin{rmk}
	We point out that Proposition \ref{K=Kperf} is the motivic version of the well known isomorphism between the absolute Galois group of a  field of characteristic $p$ and the one of its completed perfection.
\end{rmk}

The lemmas below were used in the previous proof. %

\begin{lemma}\label{2lim}If $\car K=p$ we let $\hat{K}$ be its completed perfection obtained by completing the direct limit of Frobenius maps $K_i\subset K_{i+1}$. If $\car K=0$ we fix a completion $C$ of some algebraic closure of $K$ and we 
	let $\widehat{K}\subset C$ be the completion of the union $\varinjlim K_h$ of a small directed system of  inclusions $K_h\subset K_{h'}$ of complete subfields. 
	
	We obtain the following canonical triangulated, monoidal equivalences:
	\[
	\RigDM^{\eff}(\widehat{K},\Lambda)\cong \RigDM^{\eff}(2-\varinjlim_i\RigSm/{K}_i,\Lambda)
	\]
	\[
	\RigDM(\widehat{K},\Lambda)\cong \RigDM(2-\varinjlim_i\RigSm/{K}_i,\Lambda)
	\]
	where the categories on the right are constructed by $(\Frobet,\B^1)$-localization out of the \'etale site  on the limit category defined by the base change functors  $\RigSm/K_i\ra \RigSm/K_{i+1}$. %
\end{lemma}

\begin{proof}
	We only prove the statement for the effective category, for brevity. Also, we remark that  
	the adjunction $$\LL a_{\tr}\colon\RigDA^{\eff}_{\et}(\widehat{K},\Lambda)\rightleftarrows\RigDM^{\eff}_{\et}(\widehat{K},\Lambda)\colon\RR o_{\tr}$$% 
	 induced by the  ``adding transfers'' operation can also be interpreted (by means of  \cite{vezz-DADM}) as a Bousfield localization over the set of relative Frobenius maps. Since Bousfield localizations preserve Quillen equivalences (\cite[Proposition 2.3]{hovey-sp}) it suffices to prove the following equivalence
	 	\[
	 \RigDA^{\eff}_{\et}(\widehat{K},\Lambda)\cong \RigDA^{\eff}_{\et}(2-\varinjlim_i\RigSm/{K}_i,\Lambda)
	 \]
	 where the category on the right is constructed by $(\et,\B^1)$-localization out of the \'etale site  on the limit category.

	 First, we remark that motives $\Lambda(X)$  associated to smooth, quasi-compact varieties defined over some $K_h$  form a set of compact generators of  $\RigDA^{\eff}(\hat{K},\Lambda)$. Indeed, every smooth variety $X$ is locally \'etale over the ball $\B^n_{\hat{K}}$ and therefore it locally admits a model over some $K_h$ by density and \cite[Lemma 1.1.52]{ayoub-rig}.%
	
	For the same reason, any $\et$-covering of such an $X$ can be refined with a covering coming from $2-\varinjlim\RigSm/K_h$. We deduce that the functor $\iota_*$ coming from the following Quillen adjunction
	\[
	\iota^*\colon \Ch\Psh(2-\varinjlim\RigSm/K_i,\Lambda)\leftrightarrows
	\Ch\Psh(\RigSm/\hat{K},\Lambda)\colon \iota_*
	\]
	commutes with $\et$-sheafification and hence preserves $\tau_{\et}$-weak equivalences. 
	
	We now consider $\B^1$-weak equivalences. We recall that the localization with respect to $\B^1$ has an explicit description, which is a generalization of (a cubical version of) Voevodsky's construction, see \cite[Appendix A]{ayoub-h1} and \cite[Section 3]{vezz-fw}.  It sends a presheaf $\mcF$ to $\Sing^{\B^1}\mcF$ which is the simple  complex associated to the cubical presheaf $\mcF(\cdot\times\square^\bullet)$ where $\square^n=\B^n$ and transition maps are defined by the usual projections, and by the $0$ and $1$ sections.  Since $\iota_*$ commutes with $\Sing^{\B^1}$ we deduce it also preserves $\B^1$-weak equivalences and hence it derives trivially.
	
	We now prove that $\LL\iota^*$ is fully faithful on a set of generators, say varieties $X$ with an \'etale map over a poly-disc $\B^n$ defined over some $K_h$. From Lemma \ref{approxch} and following exactly the same proof as \cite[Proposition 4.2]{vezz-fw} or \cite[Proposition 4.22]{vezz-rigidreal}  we obtain that for any smooth affinoid varieties $X,Y$   which are \'etale over some poly-discs over $K$, the canonical map between the following cubical $\Lambda$-modules %
	\[
	\Lambda\Hom_{2-\varinjlim\RigSm/K_i}(Y_{K_i}\times\square_{{K_i}}^\bullet,X_{K_i}){\ra} \Lambda\Hom_{\RigSm/\hat{K}}(Y_{\hat{K}}\times\square_{\hat{K}}^\bullet,X_{\hat{K}})
	\]
	induces a quasi-isomorphism on the associated simple complexes. 
	In other words, we obtain an isomorphism in the derived category $\Sing^{\B^1}\iota_*\LL\iota^*\Lambda(X)\cong\Sing^{\B^1}\Lambda(X)$ which implies $\iota_*\LL\iota^*\Lambda(X)\cong\Lambda(X)$. It follows that $(\iota^*,\iota_*)$ is a Quillen equivalence. We conclude that $\RigDA^{\eff}_{\et}(2-\varinjlim\RigSm/K_i,\Lambda)\cong\RigDA_{\et}^{\eff}(\hat{K},\Lambda)$. %
\end{proof}

\begin{lemma}\label{approxch}
	Under the same hypotheses of Lemma \ref{2lim}, we let  $X=\Spa R$ be a smooth affinoid variety over $K$ and we denote $R_h\colonequals R\otimes K_h$. Any element $\xi$ of $R\widehat{\otimes}\hat{K}$ which is algebraic separable over each generic point of  $\Spec\varinjlim R_h$ lies in $\varinjlim R_h$.
\end{lemma}

\begin{proof}
	We may and do assume that $X$ is geometrically connected. The proof is a simplified version of \cite[Proposition A.3]{vezz-fw} which we reproduce here briefly for the convenience of the reader. \\
	Step 1. We prove that we can restrict to an arbitrary non-empty rational subset $U$ of $\Spa R$. Fix a countable set $I$ of linearly independent elements in $R$ generating a dense subset, and complete it to a countable set $I\sqcup J$ of linearly independent elements in $R'\colonequals \mcO(U)$ generating a dense subset. By \cite[Proposition 2.7.1/3]{BGR} we can choose $I$ and $J$ in a way that the projection maps induce $K$-linear homeomorphisms  $R\cong\widehat{\bigoplus}_IK$ and $R' \cong\widehat{\bigoplus}_{I\sqcup J}K$.  Suppose now that $\xi$ lies in the intersection of $R'\cong\widehat{\bigoplus}_{I\sqcup J}K$ and $R\widehat{\otimes}\hat{K}\cong\widehat{\bigoplus}_{I} \hat{K}$ inside $R'\widehat{\otimes} \hat{K}\cong\widehat{\bigoplus}_{I\sqcup J}\hat{K}$. By the explicit description of the complete direct sum in terms of the direct product \cite[Proposition 2.1.5/7]{BGR} we deduce that $\xi$ lies in $\widehat{\bigoplus}_IK\cong R$ as wanted.
	
	Step 2. By the previous step we can make the following assumptions. The details are made explicit in \cite{vezz-fw}.
	\begin{enumerate}
		\item $\xi$ is algebraic over $R$.
		\item The sup-norm is multiplicative on $R$, $\varinjlim R_h$, $R\widehat{\otimes}\widehat{K}$, $R[\xi]$.
		\item  The sup-norm on $R[\xi]$ coincides with the norm induced by its inclusion in $R\widehat{\otimes}\widehat{K}$.
	\end{enumerate}
	Step 3. Let $\{\xi_i\}$ be the conjugates of $\xi$ different from $\xi$ in the separable closure $S$  of the completion of $\Frac R\widehat{\otimes}\hat{K}$ with respect to the sup-norm. We want to prove that this set is empty and we argue by contradiction. By a density argument and by the identity of norms of the previous step, we can fix an element $\beta$ in $\varinjlim R_h$ such that $|\xi-\beta|<\epsilon\colonequals\min\{|\xi_i-\xi|\}$ where we consider  the sup-norm in $S$. It restricts to the sup-norm on $R$ by \cite[Lemma 3.8.1/6]{BGR}. %
	Up to rescaling indexes, we can assume $\beta\in R$. 
	Any choice of an element $\xi_i$ induces a $R$-linear isomorphism $\tau_i\colon R[\xi]\cong R[\xi_i]$ which is an isometry with respect to the sup-norm. Therefore one has $$|\xi-\xi_i| \leq\max\{|\xi-\beta|, |\xi_i - \beta|\} = \max\{|\xi-\beta|, |\tau_i(\xi-\beta)|\} = |\xi-\beta| < \epsilon$$ leading to a contradiction. 
\end{proof}

\begin{lemma}\label{2lim2}Fix a directed system of strict inclusions of complete valued fields $K_i$. The canonical functors induce equivalences
	\[
	\RigDM^{\eff,\cp}( 2-\varinjlim\RigSm/{K_i},\Lambda)\cong 2-\varinjlim\RigDM^{\eff,\cp}({K}_i,\Lambda)
	\]
	\[
	\RigDM^{\cp}(2-\varinjlim\RigSm/{K_i},\Lambda)\cong 2-\varinjlim\RigDM^{\cp}({K}_i,\Lambda)
	\]
\end{lemma}

\begin{proof}
	We let $\iota_{ij}$ be the base change functors $\RigSm/K_i\ra\RigSm/K_j$ induced by the arrows of the directed system, and $\iota_i$ be the one induced by $\RigSm/K_i\ra \varinjlim\RigSm/K_i$.
	
	Fix an index $0$ of the directed system. It suffices to prove that for any two compact motives $M,N$ in a set of generators for $\RigDM(K_0,\Lambda)$ the following isomorphism holds
	\[
	\varinjlim\Hom_{\RigDM(K_i,\Lambda)}(\LL\iota_{0i}^*M,\LL\iota^*_{0i}N)\cong\Hom_{\RigDM(\varinjlim\RigSm/K_i,\Lambda)}(\LL\iota_0^*M,\LL\iota_0^*N).
	\]	
	We can finally conclude by means of \cite[Proposition 1.A.1]{ayoub-rig} to be adapted to the rigid analytic setting with transfers (see \cite[Proof of Proposition 2.5.52]{ayoub-rig}).
\end{proof}

We finally show how to obtain the left vertical equivalence of the diagram \eqref{diagQp}. We will again use the main result of the previous sections.

\begin{prop}\label{grgr+}
Suppose $K$ satisfies Assumption \ref{assK} and let $\widehat{K}$ be a Galois completed perfection of $K$ constructed in Proposition \ref{sugiu}.  The base-change functor $(\widehat{K}/K)^*$ induces triangulated, monoidal equivalences:
	\[
	\RigDM^{\eff,\gr}(K,\Lambda)\cong\RigDM^{\eff,\gr}(\widehat{K},\Lambda)\quad \RigDM^{\gr}(K,\Lambda)\cong\RigDM^{\gr}(\widehat{K},\Lambda).
	\]
\end{prop}

\begin{proof}
	We only prove the stable version of the statement for brevity. The category $\RigDM^{\gr}(K,\Lambda)$ is generated, as a triangulated category with small sums, by the motives associated to smooth affinoid varieties of good reduction, which are compact. In particular, we deduce that it suffices to show that the subcategories of compact objects (see Proposition \ref{cpt2}) $\RigDM^{\cp,\gr}(\widehat{K},\Lambda)$ and $\RigDM^{\cp,\gr}(K,\Lambda)$ are equivalent. %
	
	Suppose that $\widehat{K}$ is the completion of a Galois extension $L/K$ inducing a purely inseparable extension of residue fields. We can write $L$ as a union $\varinjlim K_h$ of finite algebraic extensions $K_h/K$ in a fixed algebraic closure containing $\widehat{K}$. %
	We denote by $\widehat{k}$ the residue field of $\widehat{K}$ which is the union of the residue fields $k_h$ of each $K_h$.
	
	Let $\mfX$ be a formal scheme over $\widehat{K}^\circ$ which is \'etale over $\mfB^n$. From  \cite[Theorem 18.1.2]{EGAIV4} and the equality $\widehat{k}\cong\varinjlim k_h$ we deduce the following equivalences of categories:
	\[
	\Et\Aff/\mfB^n_{\widehat{K}^\circ}\cong \Et\Aff/\A^n_{\widehat{k}}\cong 2-\varinjlim_h  \Et\Aff/\A^n_{{k}_h}\cong 2-\varinjlim_h \Et\Aff/\mfB^n_{{K^\circ_h}}
	\]
	so that $\mfX$ has a model over some $K^\circ_h$. In particular, we obtain that $\Lambda(\mfX_\eta)(n)$ lies in the triangulated category generated by the image of $2-\varinjlim_h\RigDM^{\cp,\gr}(K_h,\Lambda)$.  On the other hand, the set of motives of the form $\Lambda_{\tr}(\mfX_\eta)(n)$ generate the category $\RigDM^{\cp,\gr}(\widehat{K},\Lambda)$. We then deduce from Lemmas \ref{2lim} and \ref{2lim2} the following equivalence:
	\[
	2-\varinjlim_h\RigDM^{\cp,\gr}(K_h,\Lambda)\cong \RigDM^{\cp,\gr}(\widehat{K},\Lambda).
	\]
	From the diagram 
		\begin{equation} %
			\xymatrix{
				\DM(k)\cong\DA(k)\ar^{\sim}[d]&\FormDA(K^\circ)\ar[l]_-{\sim}\ar[r]^-{\LL (\cdot)_\eta}\ar[d]^{(K_h^\circ/K^\circ)^*}&\RigDA(K)\ar[r]^-{\LL a_{\tr}}&\RigDM(K)\ar[d]^{(K_h/K)^*}\\
			\DM({k_h})\cong	\DA({k_h})&\FormDA({K_h}^\circ)\ar[l]_-{\sim}\ar[r]^-{\LL (\cdot)_\eta}&\RigDA({K_h})\ar[r]^-{\LL a_{\tr}}&\RigDM({K_h})\\
			}
			\end{equation} 
			we deduce that each functor $\RigDM^{\cp,\gr}(K,\Lambda)\ra\RigDM^{\cp,\gr}(K_h,\Lambda)$ sends a set of generators to a set of generators. From the previous equivalence, we then conclude that the functor $\RigDM^{\cp,\gr}(K,\Lambda)\ra\RigDM^{\cp,\gr}(\widehat{K},\Lambda)$ also sends a set of generators to a set of generators. We now prove it is also fully faithful. To this aim, we first remark that  all compact objects in $\RigDM^{\gr}(K,\Lambda)$ are dualizable. 
	Indeed, by considering the following triangulated, monoidal functor preserving compact objects (the first equivalence follows from \cite[Appendix B]{ayoub-etale}) $$\DM(k,\Lambda)\cong\DA(k,\Lambda)\stackrel{\xi}{\ra}\RigDA(K,\Lambda)\stackrel{\LL a_{\tr}}{\ra}\RigDM(K,\Lambda)$$
	it suffices to remark that compact objects of the first category are dualizable (see for example \cite[Lecture 20]{mvw}).
	
	 We are then left to prove that $$\Hom_{\RigDM(K)}(\LL(\cdot)_\eta^*M,\Lambda)\cong \Hom_{\RigDM(\widehat{K})}(\LL(\cdot)_\eta^*M_{\widehat{K}^\circ},\Lambda)$$
	for each $M$ in $\FormDA^{\cp}(K^\circ)$. From Lemma \ref{2lim} the right hand side can be computed in the category $\RigDM(\varinjlim \RigSm/K_h)$ and by Galois descent we deduce 
	$$\Hom_{\RigDM(K)}(\LL(\cdot)_\eta^*M,\Lambda)\cong \Hom_{\RigDM(\widehat{K})}(\LL(\cdot)_\eta^*M_{\widehat{K}^\circ},\Lambda)^{\Gal(\widehat{K}/K)}.$$
	We now prove that $\Gal(\widehat{K}/K)$ acts trivially on $$\Hom_{\RigDM(\widehat{K})}(\LL(\cdot)_\eta^*M_{\widehat{K}^\circ},\Lambda)\cong \Hom_{\DM(\widehat{k})}(\LL(\cdot)_{\sigma}^*M_{\widehat{K}^\circ},\chi\Lambda).$$

	Let $K^\flat$ be the field constructed in Proposition \ref{sugiu}. By Propositions \ref{tiltgr} and Remarks \ref{chi} and \ref{chi1} we deduce that $\chi\Lambda\cong\Lambda\oplus\Lambda(-1)[-2]$ so that $$\Hom_{\DM(\widehat{k})}(\LL(\cdot)_{\sigma}^*M_{\widehat{K}^\circ},\chi\Lambda)\cong\Hom_{\DM(\widehat{k})}(M_{\widehat{k}},\Lambda\oplus\Lambda(-1)[-2])$$
	with  Galois action induced by the action on $M_{\widehat{k}}$. Since the extension $\widehat{k}/k$ is totally unseparable, we have $\DM(k)\cong\DM(\widehat{k})$ and the Galois action is therefore trivial,  proving our claim.
\end{proof}

We can then finally state and prove Ayoub's conjecture in the following form.

\begin{thm}\label{ayoubgr+}
Let $K$ be a field satisfying Assumption \ref{assK} and $\Lambda$ be a $\Q$-algebra.
\begin{enumerate}
\item Let  $\widehat{K}$ and $K^\flat$ be as in the statement of Proposition \ref{sugiu}. The canonical functors induce triangulated, monoidal equivalences:
\[
\RigDM^{\eff,\gr}(K,\Lambda)\cong\RigDM^{\eff,\gr}(\widehat{K},\Lambda)\cong\RigDM^{\eff,\gr}({K^\flat},\Lambda)
\]
\[
\RigDM^{\gr}(K,\Lambda)\cong\RigDM^{\gr}(\widehat{K},\Lambda)\cong\RigDM^{\gr}({K^\flat},\Lambda)
\]
\item Let $k$ be the residue field of $K$. There are  equivalences of categories:
\[
\Un\DM^{\eff}(k,\Lambda)\cong\RigDM^{\eff,\gr}(K,\Lambda)\qquad \Un\DM(k,\Lambda)\cong\RigDM^{\gr}({K},\Lambda)\
\]
induced by a choice of elements $\pi_i^\flat$ as in Proposition \ref{sugiu}.
\end{enumerate}
\end{thm}

\begin{proof}
The effective and the stable case are analogous, and we only consider here the latter. From our assumptions on the fields $\widehat{K}$ and $K^\flat$ (see Proposition \ref{sugiu})  Propositions \ref{K=Kperf} and \ref{grgr+} show that there exist canonical equivalences $\RigDM^{\gr}(K,\Lambda)\cong\RigDM^{\gr}(\widehat{K},\Lambda)$ and $\RigDM^{\gr}(K^\flat,\Lambda)\cong\RigDM^{\gr}(\widehat{K}^\flat,\Lambda)$. If we combine them with the tilting equivalence of Proposition \ref{tiltgr} we obtain the first claim.

The equivalence of the second claim follows  by combining the first point with  the equivalence of Ayoub $
\Un\DM(k,\Lambda)\cong\RigDM^{\gr}(K^\flat,\Lambda)$ (see Theorem \ref{ayoubgr}) induced by a choice of elements in $K^\flat$ generating its value  group.
\end{proof}

\section*{Acknowledgements}
This paper was prepared during my stay as a postdoctoral researcher at the Institut de Math\'ematiques de Jussieu-Paris Rive Gauche, financed by the Fondation Sciences Math\'ematiques de Paris. I thank these institutions for their support. I also thank Bernard Le Stum and an anonimous referee for their helpful remarks.

  \bibliographystyle{plain}

 \end{document}